\documentclass[smallextended,numbook,runningheads]{svjour3}   
\usepackage{amsmath}
\smartqed  

\usepackage{amsfonts,amssymb}
\usepackage{booktabs}

\usepackage{cite}
\usepackage{color}
\definecolor{marin}{rgb}{0.,0.3,0.7}
\usepackage[colorlinks,citecolor=marin,linkcolor=marin,urlcolor=marin,bookmarksopen,bookmarksnumbered]{hyperref}

\journalname{}
\date{ \phantom{b} \vspace{45mm}\phantom{e}}

\def\iu{{\rm i}}
\def\e{{\rm e}}
\def\d{{\rm d}}
\def\Z{{\mathbb Z}}
\def\N{{\mathbb N}}
\def\real{{\mathbb R}}
\def\eps{\varepsilon}
\def\bigo{{\mathcal O}}

\newcommand\calI{{\cal I}}

\newcommand\calK{{\cal K}}

\newcommand\calU{{\cal U}}

\newcommand\calZ{{\cal Z}}

\newcommand\bfj{{\mathbf j}}
\newcommand\bfk{{\mathbf k}}

\newcommand\bfx{{\mathbf x}}
\newcommand\bfy{{\mathbf y}}
\newcommand\bfz{{\mathbf z}}

\newcommand\bfZ{{\mathbf Z}}

\newcommand\bfomega{{\boldsymbol \omega}}

\newcommand\bfzero{{\mathbf 0}}
\newcommand\jvec{{\langle \bfj \rangle}}
\newcommand\const{{\hbox{Const}}}
\newcommand\trn{{\|\hskip-1pt|}}

\begin{document}

\title{Long-term analysis of semilinear wave equations with slowly varying wave speed
\thanks{This work has been supported by the Fonds National Suisse,
Project No.~200020-144313/1, and by Deutsche Forschungsgemeinschaft, SFB 1173 and project GA 2073/2-1.
}}
\titlerunning{Long-term analysis of semilinear wave equations with slowly varying wave speed}

\author{Ludwig Gauckler \and Ernst Hairer \and Christian Lubich}
\institute{
Ludwig Gauckler \at Institut f\"ur Mathematik, TU Berlin, Stra\ss e des 17.\,Juni 136,
D-10623 Berlin, Germany. \\
\email{gauckler@math.tu-berlin.de} 
\and
Ernst Hairer \at Section de math\'ematiques, 2-4 rue du Li\`evre,
Universit\'e de Gen\`eve, CH-1211
Gen\`eve 4,
Switzerland. \\
\email{Ernst.Hairer@unige.ch} 
\and
Christian Lubich \at Mathematisches Institut, Universit\"at T\"ubingen, Auf der Morgenstelle,
  D-72076 T\"ubingen, Germany.\\
  \email{Lubich@na.uni-tuebingen.de}}

\authorrunning{L. Gauckler, E. Hairer, and Ch. Lubich}

\date{}

\maketitle
\begin{abstract} A semilinear wave equation with slowly varying wave speed  is considered in one to three space dimensions on a bounded interval,  a rectangle or a box, respectively. It is shown that the action, which is the harmonic energy divided by the wave speed and multiplied with the diameter of the spatial domain, is an adiabatic invariant: it remains nearly conserved over long times,
 longer than any fixed power of the time scale of changes in the wave speed in the case of one space dimension, and longer than can be attained by standard perturbation arguments in the two- and three-dimensional cases. The long-time near-conservation of the action yields long-time existence of the solution. The proofs use modulated Fourier expansions in time.

\end{abstract}

\section{Introduction}
\label{sect:intro}
We consider semilinear wave equations $\partial_t^2 u = c(\eps t)^2 \Delta u + g(u)$ on a bounded spatial interval, a rectangle or a rectangular box of diameter $\ell$ with Dirichlet boundary conditions. 
The wave speed $c$ is slowly varying as $c(\eps t)$ with a small parameter $0<\eps\ll 1$. The nonlinearity $g(u)$ is cubic at $u=0$, and the small initial data  are assumed to have an energy of size $\bigo(\eps^2)$. We show long-time near-conservation of the  harmonic energy divided by the wave speed. 
Multiplied with the diameter $\ell$, this almost-conserved quantity is invariant under rescaling the spatial domain and has the physical dimension of an action,
$$
I = \frac \ell{2c}\bigl( \| \partial_t u \|_{L_2}^2 + c^2 \| \nabla_x u \|_{L_2}^2 \bigr) .
$$
The action dominates the square of the $H^1_0\times L_2$ norm of the solution $(u,\partial_t u)$. Its long-time near-conservation therefore yields long-time existence of the solution in  $H^1_0\times L_2$.

We here encounter  a situation with
\begin{itemize}
 \item a time-dependent principal operator $c(\eps t)^2\Delta$ where $c$ can vary in any given bounded interval that is bounded away from $0$,
  \item no conserved energy,
 \item fully resonant frequencies $j\pi c/\ell$ for $j=1,2,3,\dots$ in the one-dimensional case,
 \item an impenetrable thicket of resonances, almost-resonances and non-reso\-nan\-ces among the frequencies in higher dimensions.
\end{itemize}

We obtain near-conservation of the action over times $t\le \eps^{-N}$ for arbitrary ${N\ge 1}$  in the one-dimensional case (Theorem~\ref{thm:main-t}), and over times $t \le C_N\,\eps^{-3+1/N}$ for arbitrary ${N\ge 1}$  in the two- and three-dimensional cases (Theorem \ref{thm:main-t-3}).

On the one hand, our results can be viewed as an extension to a class of nonlinear wave equations of the classical adiabatic theorem, which states that a harmonic oscillator with a slowly varying frequency has the action (i.e., energy divided by the frequency) as an almost-conserved quantity over long times; see, e.g., \cite{Henrard1993} and \cite[Section 6.4]{ArnoldKozlovNeishtadt2006}.

On the other hand, our results are related to the recent literature on the long-time behaviour of nonlinear wave equations on bounded domains \cite{Bambusi2003,Bambusi2007,Bambusi2006,Bourgain1996,Bourgain2000,Cohen2008,Delort2009,Delort2009a,Delort2012,Delort2015,Delort2004,Delort2006,Gauckler2012}.
 
The tool for proving the results is a modulated Fourier expansion in time (MFE), which has previously been used in the long-term analysis of nonlinear wave equations in \cite{Cohen2008,Gauckler2012}; see also \cite{Hairer2013} for a review of MFE. The version of MFE used here is that for varying frequencies, which was  developed in \cite{Hairer2016,Sigg2009}. In this approach we do not use the canonical transformations of Hamiltonian perturbation theory, which should transform the system to a form from which the dynamical properties can be read off. With the MFE, we instead embed the system into a larger modulation system having
almost-invariants that allow us to infer the desired long-time properties.

As a referee suggests, a different approach to prove Theorem
\ref{thm:main-t} for the one-dimensional case might be to transform the considered
wave equation in the spirit of Neishtadt using action-angle variables, see, e.g.,
\cite[Chapter 6]{ArnoldKozlovNeishtadt2006} (in particular Propositions 6.3 and 6.7, which refer to a single-frequency finite-dimensional system), and then to apply an
abstract normal form result of Bambusi \& Giorgilli or a suitably adapted variant
thereof, see
\cite{BambusiGiorgilli1993,BambusiNekhoroshev1998,Bambusi1999}. The
conceptually different approach via modulated Fourier expansions that we
take here is self-contained and seems to be technically not more
complicated. This approach is the same for both, the one-dimensional 
case of Theorem \ref{thm:main-t} and the higher-dimensional cases of
Theorem \ref{thm:main-t-3}. The versatility of the approach manifests
itself also in the fact that it can be used to study related problems
for numerical discretizations; see, e.g., \cite{Hairer2013,Hairer2016}.

Our results are reminiscent of long-term results that are obtained from averaging and normal forms in other situations; see, e.g., \cite{Bambusi2003,Bambusi2005}. However, all such results for partial differential equations known to us invoke some resonance or nonresonance conditions. In the higher-dimensional case considered here, such conditions cannot be expected to be satisfied, and our Theorem~\ref{thm:main-t-3} addresses the question as to which time scales can be covered without {\it any} resonance or nonresonance condition. Theorem~\ref{thm:main-t-3}  shows that the attainable time scale by considering the interaction of any $N+1$ frequencies is a factor $\eps^{-1+1/N}$ longer than could be reached by approaches that do not take interactions between different frequencies into account. While this result is proved here using MFE, it is conceivable 
that it could also be proved using averaging and normal form techniques, but such a proof cannot be expected to be technically simpler.

Long-time almost-conservation results for finite-dimensional Hamiltonian systems without any nonresonance conditions are given in
\cite{Gauckler2013} and \cite{Bambusi2013}, with different proofs by MFE and by canonical transformations, respectively. Those results and techniques depend, however, heavily on the number of different frequencies and can therefore not be directly extended to partial differential equations. We further refer to \cite{Faou2013}, where a long-time stability result for plane waves in nonlinear Schr\"odinger equations on a torus is given with two different proofs, one using Birkhoff normal forms and  one using MFE. A detailed study of the relations between these two approaches to long-time results for Hamiltonian partial differential equations would certainly be of interest, but this is beyond the scope of this paper.

In Section 2 we give the precise formulation of the problem and state our main results. 
The proof of the result for the one-dimensional wave equation is given in Sections 3 and 4, that for two and three space dimensions in Sections 5 and 6. 

\section{Problem formulation and statement of the main results}

We consider the non-autonomous semilinear  wave equation on a  $d$-dimensional rectangular domain $Q=\prod_{i=1}^d (0,\ell_i)$, for $d\le 3$, with homogeneous Dirichlet boundary conditions: for $u=u(x,t)$ with $u=0$ on $\partial Q\times [0,T]$,
\begin{equation}\label{wave}
\partial_{t}^2 u = c(\eps t)^2 \Delta u + g(u,\eps t), \qquad x\in Q,\ t\ge 0,
\end{equation}
with a small parameter $0<\eps\ll 1$. The wave speed $c(\tau)$ is  assumed to be a smooth function of $\tau$ such that $c$ and all its derivatives are bounded for $\tau\ge 0$, and $c(\tau)\ge c_0 >0$. We consider this equation with small initial data satisfying
\begin{equation}\label{init}
\| \nabla_x u(\cdot,0)\| = \bigo(\eps), \quad \| \partial_t u(\cdot,0)\| = \bigo(\eps),
\end{equation}
where $\|\cdot\|$ denotes the $L_2(Q)^d$ or $L_2(Q)$ norm. For the nonlinearity we assume that it admits an expansion
$$
g(u,\tau) = \sum_{m\ge 1} a_m(\tau) u^{2m+1}
$$
such that the series and all its partial derivatives with respect to $\tau$ converge uniformly in $\tau$ for $|u|\le r$ with $r>0$ independent of $\eps$. For ease of presentation, we restrict our analysis to the case
$$
g(u,\tau)=a(\tau) u^3
$$
with a smooth coefficient function $a$ that is bounded for $\tau\ge 0$, as are all its derivatives. This particular nonlinearity shows all the difficulties present in the more general case.

We will show the following long-time existence results, which rely on the near-conservation of the harmonic energy divided by the wave speed,
\begin{equation}\label{action}
I(t) = \frac1{2c(\eps t)}\Bigl(\|\partial_t u(\cdot,t)\|^2 + c(\eps t)^2 \|\nabla_x u(\cdot,t)\|^2 \Bigr).
\end{equation}
Note that for initial values satisfying \eqref{init}, $I(0)=\bigo(\eps^2)$.

\begin{theorem}[One-dimensional case]
\label{thm:main-t} 
Consider the one-dimensional nonlinear wave equation \eqref{wave} with slowly time-dependent wave speed, with homogeneous Dirichlet boundary conditions and initial values satisfying \eqref{init}. Fix the integer $N\ge 1$ arbitrarily.
Under the above conditions, there exists $\eps_N>0$ such that for $\eps\le\eps_N$, the problem admits a solution $(u(\cdot,t),\partial_t u(\cdot,t))\in H^1_0(Q)\times L_2(Q)$ over long times $t\le \eps^{-N}$, and $I$
is an adiabatic invariant:
$$
|I(t)-I(0)|\le C_N\eps^3 \quad\hbox{ for }\ t\le \eps^{-N},
$$
with  $C_N$ independent of $\eps\le\eps_N$ and $t\le \eps^{-N}$.
\end{theorem}

\begin{theorem}[Two- and three-dimensional case]
\label{thm:main-t-3}
Consider the two- or three-dimensional nonlinear wave equation \eqref{wave} with slowly time-dependent wave speed, with homogeneous Dirichlet boundary conditions and initial values satisfying \eqref{init}. Fix $N\ge 1$ arbitrarily.
Under the above conditions, there exist $\eps_N>0$ and $\kappa_N>0$ such that for $\eps\le\eps_N$, the problem admits a solution $(u(\cdot,t),\partial_t u(\cdot,t))\in H^1_0(Q)\times L_2(Q)$ over times $t\le \kappa_N\eps^{-3+1/N}$, and $I$
is an adiabatic invariant: 
$$
|I(t)-I(0)|\le C_N\eps^{3} +C_N'\,t\,\eps^{5-1/N}  \quad\hbox{ for }\ t\le \kappa_N\eps^{-3+1/N},
$$
with  $C_N,C_N'$  independent of $\eps\le\eps_N$ and $t\le \kappa_N\eps^{-3+1/N}$.
\end{theorem}

The bound of Theorem~\ref{thm:main-t-3} is uniform for all rectangular domains for which $0<\ell_{\min}\le \ell_i \le \ell_{\max}$. No assumptions on resonances or non-resonances among the frequencies are made. It is the presence of almost-resonances among countably many frequencies that prevents us from covering longer time scales, in contrast to the situation of finitely many frequencies in ordinary differential equations where almost-resonances can be dealt with over much longer time scales; cf.~\cite{Bambusi2013,Gauckler2013}.

We remark that in both theorems, for the given integer $N$ the interaction of any $(N+1)$-tuples of  frequencies via the nonlinearity is taken into account in the proof.

\section*{Part I: Proof of Theorem~\ref{thm:main-t}}

\section{Modulated Fourier expansion for the short-time solution approximation}
\label{sec:mfe-t}

\subsection{Statement of result}
We consider the one-dimensional case where, without loss of generality, the interval is taken as $(0,\pi)$. 
In the course of this section we will prove the following result.

\begin{theorem}
 \label{thm:mfe}
In the situation of Theorem~\ref{thm:main-t}, the solution $u(x,t)$ of \eqref{wave} admits a modulated Fourier expansion
\begin{equation}\label{mfe2}
u(x,t)= \iu \sum_{j\in\Z} \sum_{k\in\Z} 
z_{j}^k (\eps t)\, \e^{\iu k\phi(\eps t)/\eps}\, \e^{\iu jx} + r(x,t),
\qquad 0\le t \le \eps^{-1},
\end{equation}
where the phase function $\phi(\tau)$ satisfies $\frac{\d\phi}{\d \tau}(\tau)=c(\tau)$ and the modulation functions $z_j^k(\tau)$  satisfy $z_j^{-k}=\overline{z_j^k}=-\overline{z_{-j}^{-k}}$ and are bounded for $0\le\tau\le 1$, together with any fixed number of derivatives with respect to $\tau$,  by
\begin{align*}
& \Bigl(\sum_{j\in\Z} j^2 \bigl( |z_{j}^j(\tau)| +  |z_{j}^{-j}(\tau)| \bigr)^2\Bigr)^{1/2} \le C_1\eps
 \\
& \Bigl(\sum_{j\in\Z}\Bigr( \sum_{k\in\Z} |j^2-k^2| |z_{j}^k(\tau) | \Bigr)^2 \Bigr)^{1/2} \le C_2\eps^3 .
\end{align*}
The remainder term is bounded by
\begin{equation}\label{r-bound}
\|r(\cdot, t)\|_{H^1_0} + \|\partial_t r(\cdot,t)\|_{L_2}  \le C_3(1+t)\eps^{N+2}, \qquad 0\le t \le \eps^{-1}.
\end{equation}
The constants $C_1,C_2,C_3$ are independent of $\eps$, but depend on $N$, on the  bound \eqref{init} of the initial values and on bounds of $c(\tau)$ and $a(\tau)$ and their derivatives.
\end{theorem}

\subsection{Spatial Fourier expansion}

We extend the initial values $u(x,0)$ and $\partial_t u(x,0)$ to odd functions on the
interval $[-\pi ,\pi]$. Since all terms in (\ref{wave}) are odd powers of $u$,
the solution of the equation remains an odd function
for all $t$. We consider the Fourier series
\[
u(x,t) = \iu\sum_{j\in\Z} u_j(t)\, \e^{\iu j x} 
\]
with real $u_j$, and $u_{-j}=-u_j$. In particular, $u_0=0$. The assumptions
on the initial conditions become
\begin{equation}\label{initialf}
\sum_{j\in\Z}  j^2 |u_j(0)|^2 = \bigo (\eps^2 ) ,\qquad
\sum_{j\in\Z}  \Big| \frac{\d}{\d t}u_j(0)\Big|^2 = \bigo (\eps^2 ) .
\end{equation}
The system of differential equations for the Fourier coefficients is given by 
\begin{equation}\label{wavevarf}
\frac{\d^2 u_j}{\d t^2} = - c(\eps t)^2  j^2 u_j 
- a(\eps t) \sum_{j_1+j_2+j_3=j} u_{j_1}u_{j_2}u_{j_3}  ,
\end{equation}
where the sum is over all $(j_1,j_2,j_3)$ satisfying $j_1+j_2+j_3=j$.

\subsection{Formal modulated Fourier expansion (MFE) in time}

For the Fourier coefficients of $u(x,t)$ we consider the MFE
\begin{equation}\label{mfe-uj}
u_j(t)  \approx  \sum_{k\in\Z} z_j^k (\eps t ) \,\e^{\iu k \phi (\eps t)/\eps} ,
\end{equation}
where the coefficient functions $z_j^k$ and the phase function $\phi$ are yet to be determined.
We introduce the slow time $\tau = \eps t$, and denote differentiation with
respect to $\tau$ by a dot. We insert the ansatz
(\ref{mfe-uj}) into (\ref{wavevarf}), and compare the coefficients of
$\e^{\iu k \phi (\tau )/\eps}$.  The coefficient of $\e^{\iu k \phi (\tau )/\eps} $ in
$\frac{\d^2}{\d t^2} u_j(t) = \eps^2 \frac{\d^2}{\d\tau^2}u_j(\tau /\eps )$ is
given by
\[
\eps^2 \ddot z_j^k + 2\iu k \eps \dot \phi \dot z_j^k +
\bigl( \iu k \eps \ddot \phi - k^2 \dot\phi^2 \bigr) z_j^k  .
\]
Consequently, the
functions $z_j^k (\tau )$ have to satisfy the system ($j,k\in\Z$)
\begin{eqnarray}
&&\eps^2 \ddot z_j^k + 2\iu k \eps \dot \phi \dot z_j^k +
\bigl( \iu k \eps \ddot \phi - k^2 \dot\phi^2 \bigr)z_j^k  +
j^2 c^2 z_j^k    \label{odezjkt}\\
&&\qquad\qquad\qquad ~~ = - a  
\sum_{j_1+j_2+j_3=j} \, \sum_{k_1+k_2+k_3=k} z_{j_1}^{k_1}z_{j_2}^{k_2}z_{j_3}^{k_3} . \nonumber
\end{eqnarray}
By assumption (\ref{initialf}) all $z_j^k$ will be bounded by $\bigo (\eps )$.
The dominant term for $|k|=|j|$ (obtained by neglecting the cubic
expression in $z$ and by putting $\eps =0$) thus motivates the definition
of the phase function $\phi (\tau )$ by
\begin{equation}\label{phit}
\dot \phi (\tau ) = c(\tau ) , \qquad \phi (0)=0 .
\end{equation}
The initial conditions 
yield
$$ 
u_j(0) = \sum_{k\in\Z} z_j^k(0),\qquad
\frac{\d}{\d t} u_j(0) = \sum_{k\in\Z} \Bigl( \iu k c (0)\,z_j^k(0) + \eps \dot z_j^k(0) \Bigr) .
$$

\subsection{Construction of the coefficient functions for the MFE}
\label{sect:construction}

We aim at constructing an approximate solution for the system (\ref{odezjkt})
having a small defect. For this we make an ansatz as a truncated series
in powers of $\eps$,
\begin{equation}\label{ansatz}
z_j^k (\tau ) = \sum_{l = 1}^{N+1} \eps^l z_{j,l}^k (\tau ) ,
\end{equation}
and we use the convention that $z_{j,l}^k (\tau )\equiv 0$ for $l\le 0$. 
Inserting (\ref{ansatz}) into (\ref{odezjkt}), comparing like powers of $\eps$
and using (\ref{phit}) yields
\begin{equation}\label{odezjlk}
\ddot z_{j,l-2}^k+ 2\iu k c \dot z_{j,l-1}^k +
\iu k  \dot c  z_{j,l-1}^k +
(j^2 - k^2 ) c^2 z_{j,l}^k   = g_{j,l}^k (\bfZ) 
\end{equation}
where for $\bfZ=(\bfz_1,\dots,\bfz_{l-2})$ with $\bfz_i=(z_{j,i}^k)$
\begin{equation}\label{gjlk}
g_{j,l}^k (\bfZ ) =  - a \sum_{l_1+l_2+l_3=l} \,\sum_{j_1+j_2+j_3=j}\,\sum_{k_1+k_2+k_3=k}
  z_{j_1,l_1}^{k_1}z_{j_2,l_2}^{k_2}z_{j_3,l_3}^{k_3} . 
\end{equation}
For $k\ne \pm j$, the equation (\ref{odezjlk}) represents an algebraic relation
for $z_{j,l}^k$, and for $k=\pm j$ a first order linear differential equation
for $z_{j,l-1}^k$. Initial values for this differential equation are obtained from
\begin{align}\label{initialbjl}
&\frac1\eps \,u_j(0)= \sum_{k\in\Z} z_{j,1}^k(0),\qquad
\frac1\eps\,\frac{\d}{\d t} u_j(0)   = \sum_{k\in\Z} \Bigl( \iu k c (0)\,z_{j,1}^k(0) \Bigr) 
\\
&0 = \sum_{k\in\Z} z_{j,l}^k(0),\qquad
0 = \sum_{k\in\Z} \Bigl( \iu k c (0)\,z_{j,l}^k(0) + \dot z_{j,l-1}^k(0) \Bigr) , \qquad l\ge 2.\label{initialbjl2}
\end{align}

The construction of the coefficient functions is done iteratively with increasing~$l$.
Assume that the functions $z_{j,l-2}^k$ and $z_{j,l-1}^k$ are already known
for all $j$ and all~$k$.
This is true for $l=1$. Equation (\ref{odezjlk}) then yields $z_{j,l}^k$ for $k\ne \pm j$.
The two relations of (\ref{initialbjl})--(\ref{initialbjl2}) are then a linear system for $z_{j,l}^{j}(0)$ and
$z_{j,l}^{-j} (0)$ (note that the case $j=0$ need not be considered, because $u_0(t)=0$).
With these initial values the two differential equations (\ref{odezjlk}) for $k=j$ and $k=-j$,
and $l$ replaced by $l+1$,
finally give the remaining functions $z_{j,l}^{\pm j}$.

With this construction, $z_j^k(\tau)$ of \eqref{ansatz} satisfies at $\tau=0$
\begin{align}
 \label{init-u}
 &\sum_{k\in\Z} z_j^k(0) - u_j(0) = 0 \\
 \label{init-udot}
 &\sum_{k\in\Z} \Bigl( \iu k c(0) z_j^k(0)+\eps \dot z_j^k(0) \Bigr) - \frac\d{\d t} u_j(0) = \eps^{N+2} \sum_{k\in\Z} \dot z_{j,N+1}^k(0).
\end{align}

\subsection{Bounds for the coefficient functions of the MFE}
\label{subsec:bounds-z}

Infinite sums are involved in the coupling term $g_{j,l}^k(\bfZ )$ of the
system (\ref{odezjlk}). For a rigorous analysis we have to investigate their convergence.

To bound the coefficient functions we consider for $\bfz_l = (z_{j,l}^k)_{j,k\in\Z}$ the norm
\begin{equation}\label{norm}
\trn \bfz_l \trn^2 =\sum_{j\in\Z^*}\Bigl( j^2 \bigl( |z_{j,l}^j| +  |z_{j,l}^{-j}| \bigr)^2
+ \Bigr( \sum_{k\in\Z} |j^2-k^2| |z_{j,l}^k | \Bigr)^2 \Bigr) ,
\end{equation}
where we use the notation $\Z^* = \Z \setminus \{ 0\}$.

\begin{lemma}\label{lem:boundg}
Let $\bfZ=(\bfz_1,\dots,\bfz_{l-2})$ with $\bfz_i = ( z_{j,i}^k )_{j,k\in\Z }$ and assume that
\[
\trn \bfz_i\trn \le B\quad \hbox{for}\quad i =1,\ldots ,l-2 .
\]
For the expression $ g_{j,l}^k (\bfZ )$ of \eqref{gjlk} there then exists a constant $C$ such that
\[
\Bigl(\, \sum_{j\in\Z^*} \Bigl( \sum_{k\in\Z} \big| g_{j,l}^k (\bfZ ) \big|\Bigr)^2 \Bigr)^{1/2}
\le C B^3 .
\]
\end{lemma}

\begin{proof}
We have
\[
\begin{array}{l}
\displaystyle
\sum_{j\in\Z^*} \sum_{k\in\Z} \big| g_{j,l}^k (\bfZ ) \big| \\[5mm]
\displaystyle
\le |a|\sum_{l_1+l_2+l_3=l}\Bigl( \sum_{j_1\in\Z^*, k_1\in\Z } |z_{j_1,l_1}^{k_1}| \Bigr)
 \Bigl( \sum_{j_2\in\Z^*, k_2\in\Z } |z_{j_2,l_2}^{k_2}| \Bigr)
\Bigl( \sum_{j_3\in\Z^*, k_3\in\Z } |z_{j_3,l_3}^{k_3}| \Bigr) .
\end{array}
\]
Note that the sum over $(l_1,l_2,l_3)$ is finite.
The Cauchy--Schwarz inequality and the inequality $|j|\le |j^2-k^2|$ for  $|k|\ne|j|$ yield
\begin{align*}
\sum_{j\in\Z^*} \sum_{k\in\Z }|z_{j,l}^{k}| &= \sum_{j\in\Z^* }|j|^{-1}\cdot  \sum_{k\in\Z } |j| |z_{j,l}^{k}| 
\\
&\le
\Bigl( \sum_{j\in\Z^*}j^{-2}\Bigr)^{1/2} \Bigl( \sum_{j\in\Z^*} \Bigl(\sum_{k\in\Z } |j| |z_{j,l}^{k}|\Bigr)^2\Bigr)^{1/2} \le c \trn \bfz_l \trn .
\end{align*}
The statement now follows, since the $\ell^2$ norm is bounded by the $\ell^1$ norm.
\qed
\end{proof}

\begin{lemma}\label{lem:boundsz}
The coefficient functions $\bfz_l (\tau ) = \bigl(z_{j,l}^k (\tau )\bigr)$, constructed in
Section~\ref{sect:construction}, are  bounded in the
norm \eqref{norm}: there exist constants $C_l$ such that
\[
\trn \bfz_l (\tau ) \trn \le C_l \quad \hbox{for} \quad 0\le \tau \le 1 .
\]
Bounds of the same type hold for any fixed number of derivatives of $\bfz_l (\tau )$.
\end{lemma}

\begin{proof}
Assume that $\bfz_\lambda (\tau ) $ and its derivatives up to order $N+1$ are
bounded on the interval $0\le\tau\le 1$ in the $\trn \cdot \trn$-norm for $\lambda \le l-1$.
This is true for $l=1$, because $\bfz_\lambda (\tau )\equiv 0$ for $\lambda \le 0$.

a) For $|k|\ne |j|$ it follows from (\ref{odezjlk}) that
\[
|j^2-k^2|\cdot |z_{j,l}^k| \le C \Bigl( |k|\cdot |z_{j,l-1}^k| +  |k|\cdot |\dot z_{j,l-1}^k| + 
 |\ddot z_{j,l-2}^k| + |g_{j,l}^k (\bfZ ) |\Bigr).
\]
Using $|k| \le |j^2-k^2|$, the triangle inequality for the Euclidean norm, and
Lemma~\ref{lem:boundg}, the boundedness assumption on $\bfz_{j,\lambda}^k$ and on
its derivatives (for
$\lambda \le l-1$) implies that
\begin{equation}\label{estim1}
\sum_{j\in\Z^*}\Bigl( \sum_{k\in\Z} |j^2-k^2| |z_{j,l}^k |\Bigr)^2 \le C .
\end{equation}

b) Solving the linear system (\ref{initialbjl})--(\ref{initialbjl2}) for $z_{j,l}^j(0)$ and $z_{j,l}^{-j}(0)$ yields
\[
2\iu j c(0) z_{j,l}^{\pm j}(0) = b_{j,l}^{\pm j} -
\sum_{|j|\ne |k|} \iu (k\pm j) c(0) z_{j,l}^k(0) + \sum_{j,k} \dot z_{j,l-1}^k (0),
\]
where $b_{j,l}^{\pm j} = (\d u_j/\d t(0)\pm \iu j u_j(0))/\eps$ for $l=1$, and $b_{j,l}^{\pm j} = 0$ for $l\ge 2$.
Using $|k\pm j|\le |j^2-k^2|$, the assumption on the initial values, the estimate
of part~(a) for $z_{j,l}^k (0)$, and the boundedness of $\trn \dot \bfz_{l-1}(0) \trn$, we obtain
\begin{equation}\label{estim2}
\sum_{j\in\Z^*} j^2 \bigl( \big|z_{j,l}^{ j}(0)\big| +  \big|z_{j,l}^{ -j}(0)\big| \bigr)^2 \le C .
\end{equation}

c) For $k= \pm j$, equation (\ref{odezjlk}), with $l$ augmented by $1$, yields the
differential equation for $z_{j,l}^{\pm j}$
\[
\pm 2\iu j c \dot z_{j,l}^{\pm j} \pm \iu j \dot c z_{j,l}^{\pm j} = - \ddot z_{j,l-1}^{\pm j}
+ g_{j,l+1}^{\pm j} (\bfZ,\tau ) .
\]
By the variation of constants formula we obtain, for $0\le \tau \le 1$,
\[
|j| | z_{j,l}^{\pm j} (\tau ) | \le C_1 |j| | z_{j,l}^{\pm j} (0) | + C_2 \max_{0\le\sigma \le \tau}
\Bigl( \big| \ddot z_{j,l-1}^{\pm j} (\sigma ) \big| + \big| g_{j,l+1}^{\pm j} \bigl( \bfZ (\sigma ),\sigma\bigr) \big|
\Bigr) .
\]
Using (\ref{estim2}), the boundedness of $\trn \ddot \bfz_{l-1}(0) \trn$, and
Lemma~\ref{lem:boundg}, the triangle inequality for the Euclidean norm yields, for $0\le \tau \le 1$,
\begin{equation}\label{estim3}
\sum_{j\in\Z^*} j^2 \bigl( \big|z_{j,l}^{ j}(\tau )\big| +  \big|z_{j,l}^{ -j}(\tau )\big| \bigr)^2 \le C .
\end{equation}
The estimates (\ref{estim1}) and (\ref{estim3}) prove the boundedness of
$\trn \bfz_l (\tau ) \trn $ for $\tau \in [0,1]$. The bound on the derivatives of $\bfz_l (\tau )$
is obtained in the same way after differentiating the equation (\ref{odezjlk}).
\qed
\end{proof}

It follows from the triangle
inequality that, for sufficiently small $\eps$, %
\begin{equation}\label{estimatez}
\trn \bfz(\tau) \trn \le C \eps , \qquad 0\le \tau\le 1.
\end{equation}
Moreover, it follows from the construction of Section~\ref{sect:construction}
that for $|k|\ne |j|$ we
have $z_{j,1}^k = z_{j,2}^k = 0$. This implies
\begin{equation}\label{estnondiag}
\Bigl( \sum_{j\in\Z^*}
  \Bigl(\sum_{k\in\Z} |j^2-k^2| |z_{j}^k(\tau) |\Bigr)^2 \Bigr)^{1/2}  \le C \eps^3 , \qquad 0\le \tau\le 1,
\end{equation}
which shows that the diagonal terms $z_j^j$ and $z_j^{-j}$ are dominant
in the modulated Fourier expansion (\ref{mfe-uj}). These two bounds are also
valid for any finite number of derivatives of $z_j^k$.

\subsection{Bounds for the defect}

As an approximation for the solution of (\ref{odezjkt}) we consider the truncated series \eqref{ansatz}
with coefficient functions $z_{j,l}^k (\tau ) $ constructed in Section~\ref{sect:construction}, and $\dot z_{j,N+1}^{\pm j}(\tau)\equiv 0$.
The defect, when $ z_j^k (\tau )$ is inserted into (\ref{odezjkt}), is given by
\begin{eqnarray}
d_j^k &=& \eps^2 \ddot  z_j^k + 2\iu k \eps c \dot z_j^k +
\bigl( \iu k \eps \dot c - k^2 c^2 \bigr)z_j^k  +
j^2 c^2 z_j^k    \label{defect}\\
&+&  a  \sum_{k_1+k_2+k_3=k}
\sum_{j_1+j_2+j_3=j}  z_{j_1}^{k_1}z_{j_2}^{k_2}z_{j_3}^{k_3} . \nonumber
\end{eqnarray}
By construction of the coefficient functions $z_{j,l}^k (\tau ) $ the coefficients of
$\eps^l$ vanish for $l \le N+1$. All that remains is
\begin{equation}\label{defectjk}
d_j^k = \eps^{N+2} \Bigl(\eps\ddot z_{j,N+1}^k +  \ddot z_{j,N}^k 
+2\iu k c \dot z_{j,N+1}^k +\iu k \dot c  z_{j,N+1}^k
+  a \!\!\sum_{l=N+2}^{3N+3} \!\!\eps^{l-N-2} g_{j,l}^k (\bfZ ) \Bigr)
\end{equation}
with $g_{j,l}^k (\bfZ ) $ defined in (\ref{gjlk}).

\begin{lemma}\label{lem:defect}
Under the assumptions of Theorem~\ref{thm:mfe}, there exists a constant $C_N$
such that, for $0\le \tau \le 1$, the defect is bounded by
\[
\Bigl(\, \sum_{j\in\Z^*} \Bigl( \sum_{k\in\Z} \big| d_{j}^k (\tau  ) \big| \Bigr)^2 \Bigr)^{1/2}
\le C_N \eps^{N+2} .
\]
\end{lemma}

\begin{proof}
The bound is obtained by applying the triangle inequality to (\ref{defectjk}), and by using the
bounds of Lemmas~\ref{lem:boundsz} and~\ref{lem:boundg}.
\qed
\end{proof}

\subsection{Remainder term of the MFE}
\label{subsec:remainder}
With the obtained estimate for the defect we will bound the error between the exact solution $u(\cdot,t)$ and its approximation by the MFE,
$$
\widetilde u(x,t) = \iu \sum_{j\in\Z^*} \sum_{k\in\Z} z_j^k (\eps t ) \,\e^{\iu k \phi (\eps t)/\eps} \e^{\iu jx}
$$
with $z_j^k(\eps t)$ given by \eqref{ansatz}. For this we need first to bound the solutions of the linear wave equation 
$$
\partial_t^2 w = c(\eps t)^2 \partial_x^2 w
$$
on the interval $ s \le t \le \eps^{-1}$ and initial values given at $s$.
\begin{lemma} \label{lem:U}
 The evolution family $U(t,s)$, $0\le s \le t \le \eps^{-1}$, which maps $(w(\cdot,s), \partial_t w(\cdot,s))$ to $(w(\cdot,t), \partial_t w(\cdot,t))$, is  a bounded family of linear operators on $H^1_0(0,\pi)\times L_2(0,\pi)$.
\end{lemma}

\begin{proof}
 We consider 
 $$
 I(t) = \frac1{2c(\eps t)} \Bigl( \| \partial_t w(\cdot,t) \|^2 + c(\eps t)^2  \| \partial_x w(\cdot,t) \|^2 \Bigr),
 $$
 which has the time derivative
 \begin{align*}
 \frac\d{\d t}I(t) = & - \frac{\eps \dot c(\eps t)} {2 c(\eps t)} \, I(t) 
 + \frac1{c(\eps t)} \int_0^\pi \partial_t^2 w(x,t)\, \partial_t w(x,t)\, \d x 
 \\
 & + c(\eps t) \int_0^\pi \partial_t\partial_x w(x,t)\, \partial_x w(x,t)\, \d x 
 + \eps \dot c(\eps t)  \| \partial_x w(\cdot,t) \|^2.
 \end{align*}
 On using the wave equation and partial integration, the second and third term on the right-hand side cancel.
Hence we obtain
$$
\Bigl| \frac\d{\d t} I(t) \Bigr| \le C\eps I(t)
$$
and therefore 
$$
I(t)\le \const \, I(s), \qquad 0\le s \le t \le \eps^{-1}.
$$
Since $c(\tau)$ is bounded and bounded away from $0$, this yields the result.\qed
\end{proof}

\begin{lemma} The error between the exact solution $u$ of the nonlinear wave equation and its MFE approximation $\widetilde u$ satisfies
$$
 \| \widetilde u(\cdot, t) - u(\cdot, t)\|_{H^1} + \| \partial_t  \widetilde u(\cdot, t) -  \partial_t u(\cdot, t) \|_{L_2} \le  C(1+t)\eps^{N+2}, \qquad t \le \eps^{-1}.
$$
\end{lemma}

\begin{proof}
We have
\begin{align*}
 \partial_t^2 u &= c(\eps t)^2 \partial_x^2 u + g(u,\eps t) \\
  \partial_t^2 \widetilde u &= c(\eps t)^2 \partial_x^2  \widetilde u + g( \widetilde u,\eps t) + d
\end{align*}
with
$$
d(x,t) = \iu \sum_{j\in\Z} \sum_{k\in\Z} d_j^k(\eps t) \, \e^{\iu k \phi(\eps t)/\eps} \, \e^{\iu jx}.
$$
By the variation of constants formula, the remainder term of the MFE, $r= u - \widetilde u$, satisfies
\begin{align*}
\begin{pmatrix} r(\cdot,t) \\ \partial_t r(\cdot,t) \end{pmatrix} =
& \ U(t,0)\begin{pmatrix}
       r(\cdot,0) \\ \partial_t r(\cdot,0)
      \end{pmatrix}
\\ & +
\int_0^t U(t,s) \begin{pmatrix} 0 \\ g( u(\cdot,s),\eps s) - g(\widetilde u(\cdot,s),\eps s) - d(\cdot,s) \end{pmatrix}
\, \d s .
\end{align*}
Let $0<t^*\le\eps^{-1}$ be maximal such that
\begin{equation}\label{t-star}
\| g( u(\cdot,s),\eps s) - g(\widetilde u(\cdot,s),\eps s) \|_{L_2} \le \eps \| u(\cdot,s)-\widetilde u(\cdot,s)\|_{H^1}
\quad\hbox{ for }\quad 0\le s \le t^*.
\end{equation}
Then the bound of $U$ given by Lemma~\ref{lem:U}, the bounds for the initial error \eqref{init-u}--\eqref{init-udot}, a Gronwall inequality and the bound of Lemma~\ref{lem:defect} for the defect $d(\cdot,s)$ imply, for $0\le t\le t^*$,
$$
\| r(\cdot, t) \|_{H^1} + \| \partial_t r(\cdot, t) \|_{L_2} \le C' \eps^{N+2} + C'' t \max_{0\le s\le t} \| d(\cdot,s)\|_{L_2} \le C(1+t)\eps^{N+2}.
$$
Since this bound implies that \eqref{t-star} holds with strict inequality, for sufficiently small $\eps$, the maximality of $t^*$ yields that this is possible only if $t^*$ equals the endpoint $\eps^{-1}$ of the considered time interval.\qed
\end{proof}
Combining the above lemmas concludes the proof of Theorem~\ref{thm:mfe}.

\section{Adiabatic invariant}\label{sect:adiabatic}

We show that the system for the coefficients of the modulated Fourier expansion
has an almost-invariant that is close to the adiabatic invariant of the wave equation.
Throughout this section we work with the truncated series (\ref{ansatz}).

\subsection{An almost-invariant of the MFE}

We introduce the functions
\[
y_j^k (\tau ) = z_j^k (\tau ) \, \e^{\iu k \phi (\tau )/\eps} .
\]
For the construction of the MFE we have to work with the functions $z_j^k$, which are
smooth with derivatives bounded independently of $\eps$. Here, it is more convenient
to work with the highly oscillatory functions $y_j^k$. In terms of $y_j^k$
the system (\ref{defect}) can be written as
\begin{equation}\label{odeyjk}
\eps^2 \ddot y_j^k(\tau) + j^2 c(\tau)^2 y_j^k(\tau) + \nabla_{-j}^{-k} \calU (\bfy)(\tau) = d_j^k(\tau) \, \e^{\iu k \phi(\tau) /\eps}
\end{equation}
where
\[
\calU (\bfy ) = \frac {a}4 \sum_{j_1+\ldots +j_4=0}\,  \sum_{k_1+\ldots +k_4=0} 
y_{j_1}^{k_1}y_{j_2}^{k_2}y_{j_3}^{k_3}y_{j_4}^{k_4} ,
\]
and $\nabla_{-j}^{-k}$ denotes differentiation with respect to $y_{-j}^{-k}$. The convergence
of the infinite series in the definition of $\calU (\bfy )$ follows from the proof
of Lemma~\ref{lem:boundg} provided that $\trn \bfy \trn $ is bounded.

An almost-invariant is obtained in the spirit of Noether's theorem
from the invariance property
\[
\calU \bigl( (\e^{-\iu k \theta } y_j^k )_{j,k\in\Z} \bigr) =
\calU \bigl( (y_j^k )_{j,k\in\Z} \bigr)  , \qquad \theta\in\real.
\]
Differentiation of this relation with respect to $\theta$ at $\theta =0$ yields
\[
\sum_{j,k\in \Z} (\iu k ) \,y_{-j}^{-k}\, \nabla_{-j}^{-k} \calU (\bfy ) =  0 .
\]
Furthermore, the sum
$\sum_{j,k\in\Z} (\iu k) y_{-j}^{-k} j^2 c^2 y_j^k $ vanishes, because the term for $(j,k)$
cancels with that for $(-j,-k)$. Multiplying the identity (\ref{odeyjk}) with
$(\iu k ) y_{-j}^{-k}$ and summing over all $j$ and $k$ thus yields
\begin{equation}\label{eqndefect}
\eps^2 \sum_{j,k\in \Z} (\iu k ) y_{-j}^{-k} \ddot y_j^k =
\sum_{j,k\in \Z} (\iu k ) y_{-j}^{-k} d_j^k \, \e^{\iu k \phi /\eps} =
\sum_{j,k\in \Z} (\iu k ) z_{-j}^{-k} d_j^k .
\end{equation}

\begin{theorem}\label{thm:calI}
Consider the expression
\[
\calI (\bfy  , \dot \bfy)= \eps \sum_{j,k\in\Z}  (\iu k ) y_{-j}^{-k} \dot y_j^k .
\]
Under the assumptions of Theorem~\ref{lem:boundsz}
the functions $y_j^k (\tau ) = z_j^k (\tau ) \, \e^{\iu k \phi (\tau )/\eps} $, where
$z_j^k(\tau )$ represents a truncated series (\ref{ansatz}) with coefficients
constructed in Section~\ref{sect:construction},
then satisfy, for $0\le \eps t \le 1$,
\begin{equation}\label{thmest1}
\frac {\d}{\d t} \calI \bigl( \bfy (\eps t) , \dot \bfy (\eps t) \bigr) = \bigo (\eps^{N+3})
\end{equation}
and
$$ 
\calI \bigl( \bfy (\eps t) , \dot \bfy (\eps t) \bigr)
= 2c(\eps t) \sum_{j\in\Z} j^2 \big|z_j^j(\eps t) \big|^2 + \bigo (\eps^3 ) . 
$$ 
The constant symbolised by $\bigo (\cdot )$ depends on the truncation index $N$,
but it is independent of $0<\eps \le \eps^*$ (with $\eps^*$ sufficiently
small) and of $t$ as long as $0\le \eps t \le 1$.
\end{theorem}

\begin{proof}
Differentiation of $\calI \bigl( \bfy (\eps t) , \dot \bfy (\eps t) \bigr)$
with respect to $t$ yields the lefthand expression of (\ref{eqndefect}), because
the sum $\sum_{j,k\in\Z}  (\iu k ) \dot y_{-j}^{-k} \dot y_j^k $ vanishes due to the
cancellation of the terms for $(j,k)$ and $(-j,-k)$. Applying
the Cauchy--Schwarz inequality to the righthand side of (\ref{eqndefect}),
using the estimate for the defect
(Lemma~\ref{lem:defect}) and the estimate $\trn \bfz (\tau ) \trn = \bigo (\eps )$
from (\ref{estimatez}) shows that the righthand side of (\ref{eqndefect})
is bounded by $\bigo (\eps^{N+3})$.
This proves the estimate (\ref{thmest1}).

Differentiating $y_j^k (\tau ) = z_j^k (\tau ) \, \e^{\iu k \phi (\tau )/\eps} $ with respect
to time $t$ yields
\[
\eps \dot y_j^k  = \bigl( \eps \dot z_j^k + \iu k c z_j^k \bigr) \e^{\iu k \phi /\eps} .
\]
Consequently, we have
\[
\calI (\bfy  , \dot \bfy) = - \sum_{j,k\in\Z}  c\, k^2  z_{-j}^{-k} z_j^k + \bigo (\eps^3 )
= 2\, c \sum_{j\in\Z^*} j^2 |z_j^j |^2 + \bigo (\eps^3 ).
\]
The last equality follows from (\ref{estnondiag}) 
and from the fact that
 $z_{-j}^{-k} = -\overline{z_j^k}$, which follows from
$u_{-j}=-u_j$ and 
$z_j^{-k} =  \overline {z_j^k}$. 
This proves the second statement of the theorem.
\qed
\end{proof}

\subsection{Connection with the action of the wave equation}

We consider the harmonic energy divided by the wave speed along the MFE approximation $\widetilde u(x,t)$ to the solution $u(x,t)$ as given by Theorem~\ref{thm:mfe},
\begin{equation}\label{adiabatict}
\begin{array}{rcl}
\widetilde I(t) &=& \displaystyle \frac 1{2c(\eps t)} \Bigl( \| \partial_t \widetilde u(\cdot,t) \|^2 + 
c(\eps t)^2 \|\partial_x \widetilde u(\cdot,t) \|^2 \Bigr) \\[4mm]
&=& \displaystyle  \frac 1{2c(\eps t)} \Bigl( \sum_{j\in\Z} \Bigl|\frac\d{\d t} \widetilde u_j(t)\Bigr|^2 +c(\eps t)^2 
\sum_{j\in\Z} j^2 | \widetilde u_j(t)|^2  \Bigr).
\end{array}
\end{equation}

\begin{lemma}\label{lem:action}
Let $\widetilde u_j(t) = \sum_{k\in\Z} z_j^k (\eps t ) \, \e^{\iu k \phi (\eps t )/\eps}$, where
$z_j^k (\tau )$ is the truncated series \eqref{ansatz}.
In terms of these coefficients the
action \eqref{adiabatict} satisfies
\[
\widetilde I(t)
= 2c(\eps t) \sum_{j\in\Z} j^2 \big|z_j^j(\eps t) \big|^2 + \bigo (\eps^3 ) . 
\]
\end{lemma}

\begin{proof}
Differentiating $\widetilde u_j (t)$ with respect to $t$ yields, with $\tau =\eps t$,
\[
\frac\d{\d t}  \widetilde u_j (t) = \sum_{k\in\Z} \Bigl(\eps \dot z_j^k (\tau ) + \iu k c(\tau ) z_j^k(\tau ) \Bigr)
\e^{\iu k \phi (\tau )/\eps} .
\]
From the estimate (\ref{estnondiag}) we thus obtain
\[
\sum_{j\in\Z} \Bigl|\frac\d{\d t}  u_j (t)\Bigr|^2 = c(\tau )^2 \sum_{j\in\Z} j^2 \big| z_j^j (\tau ) \e^{\iu j \phi (\tau )/\eps}
-  z_j^{-j} (\tau ) \e^{-\iu j \phi (\tau )/\eps} \big|^2 + \bigo (\eps^3 ) .
\]
Similarly, we get
\[
\sum_{j\in\Z} j^2 | u_j (t)|^2 = \sum_{j\in\Z} j^2 \big| z_j^j (\tau ) \e^{\iu j \phi (\tau )/\eps}
+  z_j^{-j} (\tau ) \e^{-\iu j \phi (\tau )/\eps} \big|^2 + \bigo (\eps^3 ) .
\]
Using the identity $|a-\overline a |^2 + |a+\overline a |^2 = 4 |a|^2$, a combination of the
last two formulas gives
\[
\sum_{j\in\Z} \Bigl|\frac\d{\d t}  u_j (t)\Bigr|^2 +c(\tau )^2 
\sum_{j\in\Z} j^2 | u_j|^2 = 4\,c(\tau )^2 \sum_{j\in\Z} j^2 \big| z_j^j (\tau ) \big|
+ \bigo (\eps^3 ) .
\]
Dividing this equation by $2c(\tau )$ proves the statement of the lemma.
\qed
\end{proof}

\subsection{Transitions in the almost-invariant}
To be able to cover a longer time interval by patching together many intervals of length $\eps^{-1}$, we need the following result.

\begin{lemma}\label{lem:trans}
 Under the conditions of Theorem~\ref{thm:mfe}, let $z_j^k(\tau)$ for $0\le\tau=\eps t\le 1$ be the coefficient functions of the MFE as in Theorem~\ref{thm:mfe} for initial data $(u(\cdot,0),\partial_t u(\cdot,0))$, and let $y_j^k(\tau)=z_j^k(\tau)\e^{\iu k\phi(\tau)/\eps} $ and $\bfy(\tau)=\bigl( y_j^k(\tau)\bigr)$. Let further
 $\widetilde\bfy(\tau)=\bigl( \widetilde y_j^k(\tau)\bigr)$ be the corresponding functions of the MFE for $1\le \tau\le 2$ to the initial data $(u(\cdot,\eps^{-1}),\partial_t u(\cdot,\eps^{-1}))$, constructed as in Theorem~\ref{thm:mfe}. Then,
 $$
 \big|\calI\bigl(\bfy(1),\dot\bfy(1)\bigr) - \calI\bigl(\widetilde\bfy(1),\dot{\widetilde\bfy}(1)\bigr)\big|\le C\eps^{N+2},
 $$
 where $C$ is independent of $\eps$.
\end{lemma}

\begin{proof}
First we note that $\bfz(\tau+1)$ contains the modulation functions that are uniquely constructed (up to $\bigo(\eps^{N+2})$) by starting from $(\widetilde u(\cdot,\eps^{-1}),\partial_t{\widetilde u}(\cdot,\eps^{-1}))$, where $\widetilde u$ is again the approximation by  the truncated modulated Fourier expansion (\ref{mfe2})
without the remainder term. On the other hand, $\widetilde\bfz(t)$ contains the modulation functions constructed by starting from the exact solution values at time $t=\eps^{-1}$. By Theorem \ref{thm:mfe} we have $u=\widetilde u+r$ with the remainder estimate (\ref{r-bound}). We thus need to estimate $\bfz- \widetilde\bfz$  at $\tau=1$ in terms of 
$\|u(\cdot,\eps^{-1})-\widetilde u(\cdot,\eps^{-1})\|_{H^1} + \|\partial_t u(\cdot,\eps^{-1})-\partial_t \widetilde u(\cdot,\eps^{-1})\|_{L_2}$.
We proceed similarly to the proof of Lemma~\ref{lem:boundsz}, taking differences in the recursions instead of direct bounds. Omitting the details, we obtain
$$
\trn \bfz(\tau) - \widetilde\bfz(\tau) \trn  \le C\eps^{N+1} \quad\hbox{ for }\quad 1 \le\tau \le 2,
$$
and bounds of the same type hold for any fixed number of derivatives of $\bfz(\tau) - \widetilde\bfz(\tau)$.
Together with the definition of $\calI$ and the bounds of 
Lemma~\ref{lem:boundsz}, this yields the stated bound.    
\end{proof}

\subsection{Long-time conservation of the adiabatic invariant}
\label{subsec:patch}

For $n=0,1,2,\dots$, let $\bfy_n(t)$ contain the summands of the modulated Fourier expansion starting from $(u(\cdot,n\eps^{-1}),\partial_t u(\cdot,n\eps^{-1}))$. As long as the adiabatic invariant satisfies $I(u(\cdot,n\eps^{-1}),\partial_t u(\cdot,n\eps^{-1}),n\eps^{-1})\le I(u(\cdot,0),\partial_t u(\cdot,0),0)+C_0\eps^2$, Theorem~\ref{thm:calI} yields for $0\le \theta \le 1$
$$
\big|\calI \bigl(\bfy_n (n+\theta),\dot \bfy_n (n+\theta)\bigr) - \calI \bigl(\bfy_n (n),\dot \bfy_n (n)\bigr) \big| \le C\eps^{N+2}.
$$
By Lemma~\ref{lem:trans},
$$
 \big| \calI\bigl(\bfy_{n}(n),\dot\bfy_{n}(n)\bigr)-\calI\bigl({\bfy}_{n-1}(n),\dot\bfy_{n-1}(n)\bigr)\big|  \le C \eps^{N+2}.
 $$
 Summing up these estimates over $n$ and applying the triangle inequality yields, for $0\le \theta \le 1$,
$$
\big|\calI\bigl(\bfy_{n}(n+\theta),\dot\bfy_{n}(n+\theta)\bigr) -  \calI \bigl(\bfy_0 (0),\dot \bfy_0 (0)\bigr) \big| \le 2(n+1)C\eps^{N+2}.
$$
By Theorem~\ref{thm:calI} and Lemma~\ref{lem:action},  we have at $t=(n+\theta)\eps^{-1}$
$$
\big|\calI \bigl(\bfy_n (n+\theta),\dot \bfy_n (n+\theta)\bigr) - \widetilde I(t)\big| \le C\eps^3,
$$
where $\widetilde I(t)$ is the action corresponding to the MFE approximation $\widetilde u$ starting from the exact solution at time $n\eps^{-1}$.
Moreover,  by the remainder estimate of Theorem~\ref{thm:mfe} and since the $H^1_0\times L_2$ norm of $(\widetilde u(\cdot,t), \partial_t \widetilde u(\cdot,t))$ is bounded by $\bigo(\eps)$, we have
$$
|\widetilde I(t) -I(t) | \le C\eps^{N+2},
$$
where $I(t)$ is the action for the solution $u(\cdot,t)$ as in \eqref{action}.
Combining these bounds at $t$ and at $0$ we obtain for $t\le\eps^{-N}$
$$
\big|I (t)-I(0))\big| \le 2C'' t \eps^{N+3} + 2C'\eps^3\le C\eps^3.
$$
This is the bound of Theorem~\ref{thm:main-t}.

\section*{Part II: Proof of Theorem~\ref{thm:main-t-3}}

We consider only the spatially three-dimensional case, since the modifications required for the two-dimensional case are obvious.

\section{Modulated Fourier expansion for the short-time solution approximation}
\label{sec:mfe-t-3}

\subsection{Spatial Fourier expansion}

We extend the initial values $u(\cdot,0)$ and $\partial_t u(\cdot,0)$ to odd functions on the
extended rectangular box $[-\ell_1,\ell_1]\times[-\ell_2,\ell_2]\times[-\ell_3,\ell_3]$. Since all terms in (\ref{wave}) are odd powers of $u$,
the solution of the equation remains an odd function
for all $t$. In this section we  write $\bfx$ instead of $x$ for the spatial variable and consider the Fourier series
\[
u(\bfx,t) = \iu\sum_{\bfj\in\Z^3} u_\bfj(t)\, \e^{\iu\bfj\circ\bfx} \quad\hbox{ with}\quad
\bfj\circ\bfx= \frac{j_1\pi x_1}{\ell_1}+ \frac{j_2\pi x_2}{\ell_2}+\frac{j_3\pi x_3}{\ell_3}
\]
for $\bfj=(j_1,j_2,j_3)$ and $\bfx=(x_1,x_2,x_3)$. We obtain real $u_\bfj$, and $u_{(-j_1,j_2,j_3)}=-u_{(j_1,j_2,j_3)}$ and similarly in the second and third component. In particular, $u_\bfj=0$ if one of the components of $\bfj$ is zero. 
The system of differential equations for the Fourier coefficients is given by 
\begin{equation}\label{wavevarf-3}
\frac{\d^2 u_\bfj}{\d t^2} = - c(\eps t)^2  \Omega_\bfj^2 u_\bfj 
- a(\eps t) \sum_{\bfj_1+\bfj_2+\bfj_3=\bfj} u_{\bfj_1}u_{\bfj_2}u_{\bfj_3}  ,
\end{equation}
where the sum is over all $\bfj_1,\bfj_2,\bfj_3\in\Z^3$ satisfying $\bfj_1+\bfj_2+\bfj_3=\bfj$, and
$\Omega_\bfj>0$ is defined by
$$
\Omega_\bfj^2 =  \Bigl( \frac{j_1\pi}{\ell_1}\Bigr)^2+ \Bigl( \frac{j_2\pi}{\ell_2}\Bigr)^2+  \Bigl( \frac{j_3\pi}{\ell_3}\Bigr)^2.
$$
The assumptions
on the initial conditions become
$$ 
\sum_{\bfj\in\Z^3}  \Omega_\bfj^2 |u_\bfj(0)|^2 = \bigo (\eps^2 ) ,\qquad
\sum_{\bfj\in\Z^3}  \Big| \frac{\d}{\d t}u_\bfj(0)\Big|^2 = \bigo (\eps^2 ) .
$$

\subsection{Statement of result}
We denote by $\Z^{*,3}$ the subset of those  $\bfj\in\Z^3$ that have all components different from zero.
We consider the linear arrangement $0<\omega_1<\omega_2<\dots$ of the different frequencies among the $\Omega_\bfj$ for  $\bfj\in\Z^{*,3}$. We let $m(\bfj)$ be the integer such that 
$$
\omega_{m(\bfj)}=\Omega_\bfj.
$$
For a sequence of integers $\bfk=(k_1,k_2,\dots)$ with only finitely many nonzero entries, we denote 
$$
\| \bfk \| = \sum_{m\ge 1} |k_m|, \qquad \bfk\cdot\bfomega= \sum_{m\ge 1} k_m\omega_m.
$$
We let $\langle \bfj \rangle=(0,\dots,0,1,0,\dots)$ be the sequence that has an entry $1$ at the $m(\bfj)$-th position and zero entries else, so that $\langle \bfj \rangle \cdot \bfomega = \omega_{m(\bfj)}=\Omega_\bfj$.

For the Fourier coefficients of $u(\cdot,t)$ we consider the MFE
\begin{equation}\label{mfe-uj-3}
u_\bfj(t)  \approx  \sum_{\bfk\in\calK_\bfj} z_\bfj^\bfk (\eps t ) \,\e^{\iu (\bfk\cdot\bfomega) \phi (\eps t)/\eps} ,
\end{equation}
where the phase function $\phi$ is given by \eqref{phit} and the modulation functions $z_\bfj^\bfk$  are to be determined. The summation is over the set
\begin{equation}\label{Kj}
\calK_\bfj = \bigl\{\, \bfk=(k_1,k_2,\dots)\in \Z^\N\,: \,  \bigl| |\bfk\cdot\bfomega|-\Omega_\bfj \bigr| \ge \eps^{1-\alpha}  \,\bigr\} \cup \{ \pm\jvec \},
\end{equation}
where we are interested in choosing a small $\alpha>0$.
This set is chosen to deal with almost-resonances: if $\bigl||\bfk\cdot\bfomega|-\Omega_\bfj \bigr| < \eps^{1-\alpha}$, then 
$\e^{\iu (\bfk\cdot\bfomega) \phi (\tau)/\eps}= w_\bfj^\bfk(\tau)\e^{\iu \Omega_\bfj \phi (\tau)/\eps}$
with
$$
w_\bfj^\bfk(\tau) = \e^{\iu ((\bfk\cdot\bfomega)-\Omega_\bfj) \phi (\tau)/\eps},
$$
where the $q$-th derivative of $w_\bfj^\bfk(\tau)$ is of magnitude $\bigo(\eps^{-q\alpha})$, so that  
$w_\bfj^\bfk(\tau)$ is changing more slowly than $\e^{\iu \Omega_\bfj \phi (\tau)/\eps}$.
\begin{theorem}
 \label{thm:mfe-3} Let the integer $N\ge 4$ be arbitrary and let $\alpha=1/N$ in the definition \eqref{Kj} of the set $\calK_\bfj$.
In the situation of Theorem~\ref{thm:main-t-3}, the solution $u(\bfx,t)$ of \eqref{wave} admits a modulated Fourier expansion 
$$ 
u(\bfx,t)= \iu\sum_{\bfj\in\Z^{*,3}} \:\sum_{\bfk\in\calK_\bfj\atop \|\bfk\|\le N+1} 
z_{\bfj}^\bfk (\eps t)\, \e^{\iu (\bfk\cdot\bfomega) \phi(\eps t)/\eps}\, \e^{\iu\, \bfj\circ\bfx} + r(\bfx,t),
\qquad 0\le t \le \eps^{-1},
$$ 
where the phase function $\phi(\tau)$ satisfies $\frac{\d\phi}{\d \tau}(\tau)=c(\tau)$ and the modulation functions $z_\bfj^\bfk(\tau)$ satisfy $z_\bfj^{-\bfk}=\overline{z_\bfj^\bfk}=-\overline{z_{-\bfj}^{-\bfk}}$ and are bounded for $0\le\tau\le 1$, together with their first and second derivatives,  by
\begin{align*}
& \Bigl(\sum_{\bfj\in\Z^{*,3}} \bigl(\Omega_\bfj ( |z_{\bfj}^{\jvec}(\tau)| +  |z_{\bfj}^{-\jvec}(\tau)| )\bigr)^2\Bigr)^{1/2} \le C_1\eps
 \\
& \Bigl(\sum_{\bfj\in\Z^{*,3}}\Bigr( \sum_{\bfk\ne\pm\jvec} (\Omega_\bfj+|\bfk\cdot\bfomega|) |z_{\bfj}^\bfk(\tau) | \Bigr)^2 \Bigr)^{1/2} \le C_2\,\eps^{2+\alpha} .
\end{align*}
The remainder term is bounded by
$$ 
\|r(\cdot, t)\|_{H^1_0} + \|\partial_t r(\cdot,t)\|_{L_2}  \le C_3\,(1+t)\,\eps^{4-\alpha}, \qquad 0\le t \le \eps^{-1}.
$$ 
The constants $C_1,C_2,C_3$ are independent of $\eps$, but depend on $N$, on the  bound \eqref{init} of the initial values and on bounds of $c(\tau)$ and $a(\tau)$ and their derivatives.
\end{theorem}

This result will be proved in the course of this section. Note that no resonance or non-resonance conditions are imposed on the frequencies $\Omega_\bfj$.

\subsection{Formal modulated Fourier expansion (MFE) in time}

We denote again differentiation with
respect to the slow time $\tau=\eps t$ by a dot. We insert the ansatz
(\ref{mfe-uj-3}) into (\ref{wavevarf-3}), and compare the coefficients of
$\e^{\iu (\bfk\cdot\bfomega) \phi (\tau )/\eps}$.  
The
functions $z_\bfj^\bfk (\tau )$ thus have to satisfy the following system: for $\bfj\in\Z^{*,3}$ and $\bfk\in\calK_\bfj$ with $\bfk\ne\pm\jvec$,
\begin{eqnarray}
&&\eps^2 \ddot z_\bfj^\bfk + 2\iu (\bfk\cdot\bfomega) \eps c \dot z_\bfj^\bfk +
\bigl( \iu (\bfk\cdot\bfomega) \eps \dot c - \bigl((\bfk\cdot\bfomega)^2-\Omega_\bfj^2\bigr) c^2 \bigr)z_\bfj^\bfk  
\label{odezjkt-k}\\
&&\qquad\qquad\qquad ~~ = -a  \sum_{\bfj_1+\bfj_2+\bfj_3=\bfj} \sum_{\bfk_1+\bfk_2+\bfk_3=\bfk} \,
  z_{\bfj_1}^{\bfk_1}z_{\bfj_2}^{\bfk_2}z_{\bfj_3}^{\bfk_3} , \nonumber
\end{eqnarray}
and for $\pm\jvec$,
\begin{eqnarray}
&&\eps^2 \ddot z_\bfj^{\pm\jvec} \pm \iu \Omega_\bfj \eps \bigl( 2c \dot z_\bfj^{\pm\jvec}  +\dot c z_\bfj^{\pm\jvec} \bigr)
\label{odezjkt-j}\\
&&\qquad ~~ = - a  \sum_{ \bfk\,:\,|(\bfk\cdot\bfomega)\mp\Omega_\bfj | < \eps^{1-\alpha} }
w_\bfj^{\bfk}
\sum_{\bfj_1+\bfj_2+\bfj_3=\bfj} 
\sum_{\bfk_1+\bfk_2+\bfk_3=\bfk} \,
 z_{\bfj_1}^{\bfk_1}z_{\bfj_2}^{\bfk_2}z_{\bfj_3}^{\bfk_3} , \nonumber
\end{eqnarray}
where the innermost sums are over all $\bfk_i\in\calK_{\bfj_i}$ with ${\bfk_1+\bfk_2+\bfk_3=\bfk}$. Note that the outer sum in \eqref{odezjkt-j} is over $\bfk$ that are not in $\calK_\bfj$ with the exception of $\pm\jvec$.
The initial conditions 
yield
$$ 
u_\bfj(0) = \sum_{\bfk\in\calK_\bfj} z_\bfj^\bfk(0),\qquad
\frac{\d}{\d t} u_\bfj(0) = \sum_{\bfk\in\calK_\bfj} \Bigl( \iu (\bfk\cdot\bfomega) c (0)\,z_\bfj^\bfk(0) + \eps \dot z_\bfj^\bfk(0) \Bigr) .
$$

\subsection{Construction of the coefficient functions for the MFE}
\label{sect:construction-3}

We aim at constructing an approximate solution for the system (\ref{odezjkt-k})--(\ref{odezjkt-j})
having a defect of magnitude $\bigo(\eps^{4-\alpha})$, which in the next section will turn out to be the permissible magnitude that yields near-conservation of the adiabatic invariant over times $t=o(\eps^{-3+\alpha})$, a time scale that we cannot improve even with a smaller defect. We make an ansatz as a truncated series
in powers of $\eps$,
\begin{equation}\label{ansatz-3}
z_\bfj^\bfk (\tau ) = \sum_{l = 1}^{N+1} \eps^l z_{\bfj,l}^\bfk (\tau ) ,
\end{equation}
for a given truncation number $N$.
It is convenient to use the convention that $z_{\bfj,l}^\bfk (\tau )\equiv 0$ for $l\le 0$ and also for $\bfk\notin\calK_\bfj$.
Inserting (\ref{ansatz-3}) into (\ref{odezjkt-k})--(\ref{odezjkt-j}) and comparing  powers of $\eps$ yields,
for $\bfj\in\Z^{*,3}$ and $\bfk\in\calK_\bfj$ with $\bfk\ne\pm\jvec$,
\begin{equation}
\ddot z_{\bfj,l-2}^\bfk + 2\iu (\bfk\cdot\bfomega) c \dot z_{\bfj,l-1}^\bfk +
 \iu (\bfk\cdot\bfomega) \dot c z_{\bfj,l-1}^\bfk - \bigl((\bfk\cdot\bfomega)^2-\Omega_\bfj^2\bigr) c^2 z_{\bfj,l}^\bfk  
 = g_{\bfj,l}^\bfk (\bfZ ) 
\label{odezjkt-kl}
\end{equation}
and for $\pm\jvec$,
\begin{equation}
\label{odezjkt-jl}
\ddot z_{\bfj,l-2}^{\pm\jvec} \pm \iu \Omega_\bfj  \bigl( 2c \dot z_{\bfj,l-1}^{\pm\jvec}  +\dot c z_{\bfj,l-1}^{\pm\jvec} \bigr)
 =   \sum_{ \bfk\,:\,|(\bfk\cdot\bfomega)\mp\Omega_\bfj | < \eps^{1-\alpha} } 
w_\bfj^\bfk g_{\bfj,l}^\bfk (\bfZ ), 
\end{equation}
where for $\bfZ=(\bfz_1,\dots,\bfz_{l-2})$ with $\bfz_i=(z_{\bfj,i}^\bfk)$,
\begin{equation}\label{gjlk-3}
g_{\bfj,l}^\bfk (\bfZ ) =  -a \sum_{l_1+l_2+l_3=l} \,\sum_{\bfj_1+\bfj_2+\bfj_3=\bfj}\,\sum_{\bfk_1+\bfk_2+\bfk_3=\bfk}
  z_{\bfj_1,l_1}^{\bfk_1}z_{\bfj_2,l_2}^{\bfk_2}z_{\bfj_3,l_3}^{\bfk_3} . 
\end{equation}
For $\bfk\ne \pm \jvec$, equation (\ref{odezjkt-kl}) represents a linear equation
for $z_{\bfj,l}^\bfk$, and (\ref{odezjkt-jl}) is a first order linear differential equation
for $z_{\bfj,l-1}^{\pm\jvec}$. Initial values for this differential equation are obtained from
\begin{align}\label{initialbjl-3}
&\frac1\eps \,u_\bfj(0)= \sum_{\bfk\in\calK_\bfj} z_{\bfj,1}^\bfk(0),\qquad
\frac1\eps\,\frac{\d}{\d t} u_\bfj(0)   = \sum_{\bfk\in\calK_\bfj} \Bigl( \iu (\bfk\cdot\bfomega) c (0)\,z_{\bfj,1}^\bfk(0) \Bigr) 
\\
&0 = \sum_{\bfk\in\calK_\bfj} z_{\bfj,l}^\bfk(0),\qquad
0 = \sum_{\bfk\in\calK_\bfj} \Bigl( \iu (\bfk\cdot\bfomega) c (0)\,z_{\bfj,l}^\bfk(0) + \dot z_{\bfj,l-1}^\bfk(0) \Bigr) , \qquad l\ge 2.\label{initialbjl2-3}
\end{align}

The construction of the coefficient functions is done iteratively with increasing~$l$, 
as in Section~\ref{sect:construction}. We note that $z_{\bfj,l}^\bfk$ can differ from zero only for $\|\bfk\|\le l$.
Moreover, the initial values $z_\bfj^\bfk(0)$ of \eqref{ansatz-3} satisfy
\begin{align}
 \label{init-u-3}
 &\sum_{\bfk\in\calK_\bfj} z_\bfj^\bfk(0) - u_\bfj(0) = 0 \\
 \label{init-udot-3}
 &\sum_{\bfk\in\calK_\bfj} \Bigl( \iu (\bfk\cdot\bfomega) c(0) z_\bfj^\bfk(0)+\eps \dot z_\bfj^\bfk(0) \Bigr) - \frac\d{\d t} u_\bfj(0) = \eps^{N+2} \sum_{\bfk\in\calK_\bfj} \dot z_{\bfj,N+1}^\bfk(0).
\end{align}

\subsection{Bounds for the coefficient functions of the MFE}

We denote by $\calZ$ the space of all $\bfz = (z_{\bfj}^\bfk)_{\bfj\in\Z^{*,3},\bfk\in\calK_\bfj,\|\bfk\|\le N+1}$ with finite norm
\begin{equation}
\trn \bfz \trn^2 = \sum_{\bfj\in\Z^{*,3}}\Bigl( \!\!\sum_{\bfk\in\calK_\bfj\atop \|\bfk\|\le N+1} \bigl(|\bfk\cdot\bfomega|+\Omega_\bfj \bigr)
|z_{\bfj}^\bfk | \Bigr)^2  .
\label{norm-3}
\end{equation}

\begin{lemma}\label{lem:boundg-3}
For $\bfz_i=(z_{\bfj,i}^\bfk)\in\calZ$ $(i=1,2,3)$ we let, for $\bfj\in\Z^{*,3}$ and $\bfk=(k_1,k_2,\dots)\in\Z^\N$,
$$
h_\bfj^\bfk(\bfz_1,\bfz_2,\bfz_3) = \sum_{\bfj_1+\bfj_2+\bfj_3=\bfj}\,\sum_{\bfk_1+\bfk_2+\bfk_3=\bfk}
  z_{\bfj_1,1}^{\bfk_1}z_{\bfj_2,2}^{\bfk_2}z_{\bfj_3,3}^{\bfk_3} .
$$
Then,
$$
\Bigl( \sum_{\bfj\in\Z^{*,3}} \Bigl( \sum_{\bfk\in\Z^\N} \big|h_\bfj^\bfk(\bfz_1,\bfz_2,\bfz_3)\big|\Bigr)^2 \Bigr)^{1/2} 
\le
C \,\trn \bfz_1 \trn\, \trn \bfz_2 \trn\,\trn \bfz_3 \trn.
$$
\end{lemma}

\begin{proof} By the Parseval formula, the left-hand side of the desired inequality is bounded by the $L_2(Q)$ norm of the function
$
  f_{1}(\bfx) \, f_{2}(\bfx)  \, f_{3}(\bfx) ,
$
where $f_i(\bfx)$ is the function with $\bfj$-th Fourier coefficient $\sum_{\bfk\in\Z^\N} |z_{\bfj,i}^\bfk|$.
We then have
$$
\| f_{1} f_{2} f_{3} \|_{L_2} \le \| f_{1} \|_{L_6} \, \| f_{2} \|_{L_6} \, \| f_{3} \|_{L_6}
\le C \,  \| f_{1} \|_{H^1} \, \| f_{2} \|_{H^1} \, \| f_{3} \|_{H^1},
$$
where we have used the H\"older inequality and the Sobolev embedding $H^1(Q)\subset L^6(Q)$, valid for dimension $d\le 3$. We further have
$$
\| f_i \|_{H^1} = \Bigl(\sum_{\bfj\in\Z^{*,3} } \Bigl( \sum_{\bfk\in\Z^\N} \Omega_\bfj\,|z_{\bfj,i}^\bfk| \Bigr)^2\Bigr)^{1/2} \le \trn {\bfz}_i \trn,
$$
which yields the result. \qed
\end{proof}

We note that $g_{\bfj,l}^\bfk(\bfZ)$ of \eqref{gjlk-3} is given, for $\bfZ=(\bfz_1,\dots,\bfz_{l-2})$, by the finite sum
$$
g_{\bfj,l}^\bfk(\bfZ) = -a \sum_{l_1+l_2+l_3=l} h_\bfj^\bfk(\bfz_{l_1},\bfz_{l_2},\bfz_{l_3}).
$$

Since we obtain different bounds for diagonal coefficient functions $z_{\bfj}^{\jvec}$ and non-diagonal coefficient functions $z_{\bfj}^{\bfk}$ with $\bfk\ne\pm\jvec$, we split
$$
\trn \bfz \trn^2 = |\bfz|_{\text{diag}}^2 + |\bfz|_{\text{off-diag}}^2,
$$
where
\begin{align*}
|\bfz|_{\text{diag}}^2 &= \sum_{\bfj\in\Z^{*,3}}\Bigl( 2\Omega_\bfj \bigl( |z_{\bfj}^{\jvec}| + |z_{\bfj}^{-\jvec}| \bigr)\Bigr)^2
\\
|\bfz|_{\text{off-diag}}^2 &= \sum_{\bfj\in\Z^{*,3}}\Bigl(\sum_
{\bfk\in\calK_\bfj, \bfk\ne\pm\jvec \atop \|\bfk\|\le N+1} 
\bigl(|\bfk\cdot\bfomega|+\Omega_\bfj \bigr)
|z_{\bfj}^\bfk | \Bigr)^2  .
\end{align*}

\begin{lemma}\label{lem:boundsz-3}
Under the conditions of Theorem~\ref{thm:mfe-3}, the seminorms 
$|\bfz_l^{(q)}|_{\text{\rm diag}}$ and $ |\bfz_l^{(q)}|_{\text{\rm off-diag}}$ of the
$q$-th derivative of the coefficient functions $\bfz_l=(z_{\bfj,l}^\bfk)$ constructed in Section~\ref{sect:construction-3} are bounded for $0<\alpha\le \min(\frac14,\frac1{(l-1)})$ as stated in the table below, uniformly for $0\le\tau\le 1$. For each $(l,q)$, the  entry in the first table gives the bound for $|\bfz_l^{(q)}|_{\text{\rm diag}}$, and that in the second table for $ |\bfz_l^{(q)}|_{\text{\rm off-diag}}$, up to a constant independent of~$\eps$. In particular, the coefficient functions $\bfz=(z_\bfj^\bfk)$ of \eqref{ansatz-3} satisfy the bounds $|\bfz|_{\text{\rm diag}}=\bigo(\eps)$ and $|\bfz|_{\text{\rm off-diag}}=\bigo(\eps^{2+\alpha})$.
\end{lemma}
\begin{center}
\begin{tabular}{c | c c c c }
\toprule
       &  $q=0$ & $q=1$ & $q=2$ & $q>2$ \\
\midrule
$l=1$ & $\eps^0$   & $\eps^0$ & $\eps^0$ & $\eps^0$    \\
$l=2$ &  $\eps^0$ & $\eps^0$ &  $\eps^{-\alpha}$   & $\eps^{-(q-1)\alpha}$ \\
$l=3$ &  $\eps^{-(1-\alpha)}$ & $\eps^{-(1-\alpha)}$ &  $\eps^{-(1-\alpha)}$   & $\eps^{-(1-\alpha)}+\eps^{-q\alpha}$ \\
$l\ge 4$ &  $\eps^{-(l-2) (1-\alpha )}$ & $\eps^{-(l-2)(1-\alpha)}$ &  $\eps^{-(l-2)(1-\alpha)}$   & $\eps^{-(l-2)(1-\alpha )}\hfill$ \\
 & & & & $+ \eps^{-(l-4)(1-\alpha)-(q+l-3)\alpha }$ \\
\bottomrule
\end{tabular}
\vskip 1mm
{Bounds for diagonal coefficient functions $z_{\bfj,l}^{\pm\jvec}$.}
\end{center}
\smallskip
\begin{center}
\begin{tabular}{c | c c c c }
\toprule
       &  $q=0$ & $q=1$ & $q=2$ & $q>2$ \\
\midrule
$l=1$ & $0$   & $0$ & $0$ & $0$    \\
$l=2$ & $0$   & $0$ & $0$ & $0$    \\
$l=3$ &  $\eps^{-(1-\alpha)}$ & $\eps^{-(1-\alpha)}$ &  $\eps^{-(1-\alpha)}$   & $\eps^{-(1-\alpha)}$ \\
$l\ge 4$ &  $\eps^{-(l-2) (1-\alpha )}$ & $\eps^{-(l-2)(1-\alpha)}$ &  $\eps^{-(l-2)(1-\alpha)}$   & $\eps^{-(l-2)(1-\alpha )}\hfill$ \\
 & & & & $ + \eps^{-(l-3)(1-\alpha ) - (q+l-5)\alpha }$\\
\bottomrule
\end{tabular}
\vskip 1mm
{Bounds for off-diagonal coefficient functions $z_{\bfj,l}^\bfk$ with $\bfk\ne\pm\jvec$.}
\end{center}

\begin{proof}
 We work with \eqref{odezjkt-kl} for the off-diagonal coefficients $z_{\bfj,l}^\bfk$ for $\bfk\in\calK_\bfj$ with $\bfk\ne\pm\jvec$, with \eqref{odezjkt-jl} for the diagonal coefficients $z_{\bfj,l}^{\pm\jvec}$, and with \eqref{initialbjl-3}--\eqref{initialbjl2-3} for the initial values. Factorizing 
 $$
 (\bfk\cdot\bfomega)^2-\Omega_\bfj^2= (|\bfk\cdot\bfomega|+\Omega_\bfj)(|\bfk\cdot\bfomega|-\Omega_\bfj)
 $$
 and using that $\bigl| |\bfk\cdot\bfomega|-\Omega_\bfj \bigr| \ge \eps^{1-\alpha}$ for
 $\bfk\in\calK_\bfj$, we obtain for $\bfk\ne\pm\jvec$
 \begin{equation}\label{est-zjlk}
 (|\bfk\cdot\bfomega|+\Omega_\bfj)\, |z_{\bfj,l}^\bfk| \le C\eps^{-(1-\alpha)} \Bigl(
 |\bfk\cdot\bfomega|\cdot |z_{\bfj,l-1}^\bfk| + |\bfk\cdot\bfomega|\cdot |\dot z_{\bfj,l-1}^\bfk| +
 |\ddot z_{\bfj,l-2}^\bfk| + |g_{\bfj,l}^\bfk(\bfZ)|\Bigr)
 \end{equation}
 and similar bounds for the derivatives of $z_{\bfj,l}^\bfk$, where the functions on the right-hand side are replaced by their corresponding derivatives. 
 
 Solving the linear systems \eqref{initialbjl-3}--\eqref{initialbjl2-3} for $z_{\bfj,l}^{\jvec}(0)$ yields
\begin{equation}\label{est-zjlj-0}
  \Omega_\bfj z_{\bfj,l}^{\pm\jvec}(0) = \frac1{2\iu c(0)}\Bigl(b_{\bfj,l}^{\pm\jvec} -\! \sum_{\bfk\ne\pm\jvec} \iu 
 ((\bfk\cdot\bfomega)\pm\Omega_\bfj)\,c(0)\, z_{\bfj,l}^{\bfk}(0) + \!
 \sum_{\bfk\in\calK_j} \dot z_{\bfj,l-1}^\bfk(0) \Bigr),
 \end{equation}
 where $b_{\bfj,l}^{\pm\jvec} = (\d u_\bfj/\d t(0) \pm \iu \Omega_\bfj u_\bfj(0))/\eps$ for $l=1$, and $b_{\bfj,l}^{\pm\jvec} = 0$ else.
 We have 
 $$
 \sum_{\bfj\in\Z^{*,3}} |b_{\bfj,l}^{\pm\jvec}|^2 \le C
 $$
 by the assumption \eqref{init} on the initial values.
 
 The linear differential equation \eqref{odezjkt-jl} with $l-1$ replaced by $l$ becomes
 \begin{equation}
\label{ode-diag}
\pm \iu \Omega_\bfj  \bigl( 2c \dot z_{\bfj,l}^{\pm\jvec}  +\dot c z_{\bfj,l}^{\pm\jvec} \bigr)
 =   -\ddot z_{\bfj,l-1}^{\pm\jvec} +\sum_{ \bfk\,:\,|(\bfk\cdot\bfomega)\mp\Omega_\bfj | < \eps^{1-\alpha} } 
w_\bfj^\bfk g_{\bfj,l+1}^\bfk (\bfZ ), 
\end{equation} 
Using the variation-of-constants formula and a partial integration yields, for $0\le\tau\le 1$,
 \begin{align}
 \Omega_\bfj \big| z_{\bfj,l}^{\pm\jvec}(\tau)\big| &\le  C_1  \Omega_\bfj \big| z_{\bfj,l}^{\pm\jvec}(0)\big| \label{est-zjlj}
 \\ 
 &\ + 
 C_2 \max_{0\le \sigma\le \tau} \Bigl(
 \big| \dot z_{\bfj,l-1}^{\pm\jvec}(\sigma)\big| +
 \sum_{ \bfk\,:\,|(\bfk\cdot\bfomega)\mp\Omega_\bfj | < \eps^{1-\alpha} } 
\bigl|w_\bfj^\bfk(\sigma) \,g_{\bfj,l+1}^\bfk (\bfZ(\sigma) )\bigr| 
\Bigr).\nonumber
\end{align}
and similar bounds for the derivatives of $z_{\bfj,l}^{\pm\jvec}$. Note that the $q$-th derivative of $w_\bfj^\bfk(\tau)$ is bounded by $\bigo(\eps^{-q\alpha})$. 
 
With these tools we can estimate the coefficient functions $z_{\bfj,l}^\bfk$ and their derivatives for one $l$ after the other. For $l\le 0$, all $z_{\bfj,l}^\bfk$ are zero by definition.

 $l=1$: The off-diagonal coefficients are zero, because $g_{\bfj,l}^\bfk(\bfZ)\equiv 0$ for $l=1$ (and also for $l=2$). 
 By \eqref{est-zjlj-0} we obtain $|\bfz_1(\tau)|_\text{diag} \le C$ for $\tau=0$, and by \eqref{est-zjlj} for all $0\le\tau\le 1$. Using \eqref{ode-diag} we obtain the same bound for any finite number of derivatives of $\bfz_1$.
 
 $l=2$: The off-diagonal coefficients are still zero. Using the bound for $\dot\bfz_1$ in \eqref{est-zjlj-0}, we obtain $|\bfz_2(\tau)|_\text{diag} \le C$ for $\tau=0$, and by \eqref{est-zjlj} for all $0\le\tau\le 1$. Using \eqref{ode-diag}, which now contains non-vanishing $g_{\bfj,3}^\bfk$ with factors $w_\bfj^\bfk$, we find that the $q$-th derivative of $\bfz_2$ contains the $(q-1)$-th  derivative of $w_\bfj^\bfk$, which is $\bigo(\eps^{-(q-1)\alpha})$. Using Lemma~\ref{lem:boundg-3}, we thus
 obtain  $|\bfz_2^{(q)}(\tau)|_\text{diag} \le C\eps^{-(q-1)\alpha}$.
 
$l=3$: By \eqref{est-zjlk}, by the bound for $\bfz_1$ and its derivatives and by Lemma~\ref{lem:boundg-3} we obtain that $\bfz_3$ and its derivatives satisfy $|\bfz_3^{(q)}|_{\text{off-diag}}=\bigo(\eps^{-(1-\alpha)})$ for all $q\ge 0$.
By \eqref{est-zjlj-0}, the initial value for the diagonal part of $\bfz_3$ is bounded by
$|\bfz_3(\tau)|_{\text{diag}}=\bigo(\eps^{-(1-\alpha)})$ at $\tau=0$, and \eqref{est-zjlj} then gives the same bound for all $\tau\le 1$. Formula \eqref{ode-diag} and its differentiated versions then yield the bound $|\bfz_3^{(q)}(\tau)|_{\text{diag}}=\bigo(\eps^{-(1-\alpha)}+\eps^{-q\alpha})$ for $q\ge 1$.
 
$l\ge 4$: The same arguments as before yield the bounds of the lemma. 
\qed
\end{proof}

\subsection{Bounds for the defect and the remainder}

In the following we choose $N\ge 4$ arbitrarily in \eqref{ansatz-3} and $\alpha=1/N$ in \eqref{mfe-uj-3}--\eqref{Kj}.
Lemma~\ref{lem:boundsz-3} then shows that $\eps^{N+2} z_{\bfj,N+1}^\bfk$ and $\eps^{N+2}\dot z_{\bfj,N+1}^\bfk$ for $\bfk\ne\pm\jvec$ are both of magnitude $\bigo(\eps^{4-\alpha})$. As we will see in a moment, these terms are the dominating terms in the defect.
For the diagonal entries $z_{\bfj,l}^{\pm \jvec}(\tau)$ only the initial value is constructed
   for $l=N+1$ and the function is taken to be constant in time, because 
of the shifted index $l-1$ in
\eqref{odezjkt-jl}.
The defect, when $ z_\bfj^\bfk (\tau )$ is inserted into (\ref{odezjkt}), is given for $\bfk\ne\pm\jvec$ by
\begin{align}
d_\bfj^\bfk &= \eps^2 \ddot  z_\bfj^\bfk + 2\iu (\bfk\cdot\bfomega) \eps c \dot z_\bfj^\bfk +
\bigl( \iu (\bfk\cdot\bfomega) \eps \dot c - (\bfk\cdot\bfomega)^2 c^2 \bigr)z_\bfj^\bfk  +
\Omega_\bfj^2 c^2 z_\bfj^\bfk    \label{defect-jk}\\
&  + a 
\sum_{\bfj_1+\bfj_2+\bfj_3=\bfj}  \sum_{\bfk_1+\bfk_2+\bfk_3=\bfk} z_{\bfj_1}^{\bfk_1}z_{\bfj_2}^{\bfk_2}z_{\bfj_3}^{\bfk_3}  \nonumber
\end{align}
and for $\pm\jvec$ by
\begin{eqnarray}
d_\bfj^{\pm\jvec}&=&\eps^2 \ddot z_\bfj^{\pm\jvec} \pm \iu \Omega_\bfj \eps \bigl( 2c \dot z_\bfj^{\pm\jvec}  +\dot c z_\bfj^{\pm\jvec} \bigr)
\label{defect-jj}\\
&&  +\ a  \sum_{ \bfk\,:\,|(\bfk\cdot\bfomega)\mp\Omega_\bfj | < \eps^{1-\alpha} }
w_\bfj^{\bfk}
\sum_{\bfj_1+\bfj_2+\bfj_3=\bfj} 
\sum_{\bfk_1+\bfk_2+\bfk_3=\bfk} \,
 z_{\bfj_1}^{\bfk_1}z_{\bfj_2}^{\bfk_2}z_{\bfj_3}^{\bfk_3} . \nonumber
\end{eqnarray}
By construction of the coefficient functions $z_{\bfj,l}^\bfk $, the coefficients of
$\eps^l$ vanish for $l \le N+1$. All that remains is, for $\bfk\ne\pm\jvec$,
\begin{align*} 
d_\bfj^\bfk = \eps^{N+2} \Bigl(&\eps\ddot z_{\bfj,N+1}^\bfk +  \ddot z_{\bfj,N}^\bfk 
+2\iu (\bfk\cdot\bfomega) c \dot z_{\bfj,N+1}^\bfk\\ &+\iu (\bfk\cdot\bfomega) \dot c  z_{\bfj,N+1}^\bfk
+  a \!\!\sum_{l=N+2}^{3N+3} \!\!\eps^{l-N-2} g_{\bfj,l}^\bfk (\bfZ ) \Bigr)
\end{align*} 
and for $\pm\jvec$,
$$ 
d_\bfj^{\pm\jvec} = \eps^{N+2} \Bigl(  \ddot z_{\bfj,N}^{\pm\jvec}
 \pm \iu \Omega_\bfj \dot c  z_{\bfj,N+1}^{\pm\jvec}
+  a \!\!\sum_{l=N+2}^{3N+3} \!\!\!\!\eps^{l-N-2} \sum_{ \bfk\,:\,|(\bfk\cdot\bfomega)\mp\Omega_\bfj | < \eps^{1-\alpha} } \!\!\!\!
w_\bfj^{\bfk} g_{\bfj,l}^\bfk (\bfZ ) \Bigr)
$$ 
with $g_{\bfj,l}^\bfk (\bfZ ) $ defined in (\ref{gjlk-3}).

\begin{lemma}\label{lem:defect-3}
Consider the approximation \eqref{ansatz-3} with arbitrary $N\ge 4$ and  \eqref{mfe-uj-3}--\eqref{Kj}
with $\alpha=1/N$.
Under the assumptions of Theorem~\ref{thm:mfe-3}, the defect is bounded, for $0\le \tau \le 1$, by
\[
\Bigl(\, \sum_{\bfj\in\Z^{*,3}} \Bigl( \sum_{\bfk\in\calK_j} \big| d_{\bfj}^\bfk (\tau  ) \big| \Bigr)^2 \Bigr)^{1/2}
\le C \eps^{4-1/N} ,
\]
where $C$ is independent of $\eps$ and $0\le\tau\le 1$, but depends on $N$.
\end{lemma}

\begin{proof}
The bound is obtained by using the above formulas for the defect and the
bounds of Lemma~\ref{lem:boundsz-3}, and Lemma~\ref{lem:boundg-3} for bounding the nonlinearity.
\qed
\end{proof}

We remark that the choice $N=4$ and $\alpha=1/2$ yields a smaller bound $\bigo(\eps^{9/2})$ for the defect. Our interest here is, however, to obtain the stated bound for arbitrarily small $\alpha>0$.

Equations \eqref{init-u-3}--\eqref{init-udot-3} and Lemma~\ref{lem:boundsz-3} also yield that the error in the initial values is bounded by $\widetilde u(\cdot,0)-u(\cdot,0)=0$ and $\| \partial_t{\widetilde u}(\cdot,0)-\partial_t u(\cdot,0)\|_{L^2} = \bigo(\eps^{4-\alpha})$. By the same argument as in Section~\ref{subsec:remainder}, it follows that the remainder term $(r,\partial_t r)$ of the MFE is bounded in $H^1_0(Q)\times L^2(Q)$ by $C(1+t)\eps^{4-\alpha}$ for $t\le\eps^{-1}$. This completes the proof of Theorem~\ref{thm:mfe-3}.

\section{Adiabatic invariant}\label{sect:adiabatic-3}

We show that the almost-invariant for the coefficients of the modulated Fourier expansion
extends from the one- to the three-dimensional case, albeit with a larger error in the near-conservation property.
Throughout this section we work with the truncated series (\ref{ansatz-3}) with arbitrary $N\ge 4$ in the  MFE \eqref{mfe-uj-3}--\eqref{Kj}
with $\alpha=1/N$.

\subsection{An almost-invariant of the MFE}

For $\bfj\in\Z^{*,3}$ and $\bfk\in\calK_\bfj$ we introduce the functions
\[
y_\bfj^\bfk (\tau ) = z_\bfj^\bfk (\tau ) \, \e^{\iu(\bfk\cdot\bfomega) \phi (\tau )/\eps} .
\]
In terms of $y_\bfj^\bfk$, for $\bfk\ne\pm\jvec$
 (\ref{defect-jk}) can be rewritten as
\begin{equation}\label{odeyjk-jk}
\eps^2 \ddot y_\bfj^\bfk(\tau) +  c(\tau)^2 \Omega_\bfj^2 y_\bfj^\bfk(\tau) + \nabla_{-\bfj}^{-\bfk} \calU (\bfy)(\tau) = d_\bfj^\bfk(\tau) \, \e^{\iu (\bfk\cdot\bfomega) \phi(\tau) /\eps}
\end{equation}
where
\[
\calU (\bfy ) = \frac {a}4  \,\sum_{\bfj_1+\ldots +\bfj_4=0}\,\sum_{\bfk_1+\ldots +\bfk_4=0}
y_{\bfj_1}^{\bfk_1}y_{\bfj_2}^{\bfk_2}y_{\bfj_3}^{\bfk_3}y_{\bfj_4}^{\bfk_4} ,
\]
and $\nabla_{-\bfj}^{-\bfk}$ denotes differentiation with respect to $y_{-\bfj}^{-\bfk}$. The convergence
of the infinite series in the definition of $\calU (\bfy )$ follows from  Lemma~\ref{lem:boundg-3} provided that $\trn \bfy \trn $ is bounded.
Equation (\ref{defect-jj}) can be written
\begin{equation}\label{odeyjk-jj}
\eps^2 \ddot y_\bfj^{\pm\jvec}(\tau) +  c(\tau)^2 \Omega_\bfj^2 y_\bfj^{\pm\jvec}(\tau) + \nabla_{-\bfj}^{\mp\jvec} \calU (\bfy)(\tau) = (d_\bfj^{\pm\jvec}(\tau) +e_\bfj^{\pm\jvec}(\tau))\, \e^{\pm\iu \Omega_\bfj \phi(\tau) /\eps},
\end{equation}
where
$$
e_\bfj^{\pm\jvec} =  -a  \sum_{ \bfk\ne\pm\jvec\,:\,|(\bfk\cdot\bfomega)\mp\Omega_\bfj | < \eps^{1-\alpha} }
\sum_{\bfj_1+\bfj_2+\bfj_3=\bfj} 
\sum_{\bfk_1+\bfk_2+\bfk_3=\bfk} \,
 y_{\bfj_1}^{\bfk_1}y_{\bfj_2}^{\bfk_2}y_{\bfj_3}^{\bfk_3}.
$$
The invariance property
\[
\calU \bigl( (\e^{-\iu (\bfk\cdot\bfomega) \theta } y_\bfj^\bfk )_{\bfj\in\Z^{*,3},\bfk\in\calK_\bfj} \bigr) =
\calU \bigl( (y_\bfj^\bfk )_{\bfj\in\Z^{*,3},\bfk\in\calK_\bfj} \bigr)  , \qquad \theta\in\real,
\] 
yields, like in Section~\ref{sect:adiabatic},
\begin{equation}\label{U-term-vanish}
\sum_{\bfj\in\Z^{*,3}}\sum_{\bfk\in\calK_\bfj} \iu (\bfk\cdot\bfomega) \,y_{-\bfj}^{-\bfk} \,\nabla_{-\bfj}^{-\bfk} \calU (\bfy ) =  0 .
\end{equation}
Moreover, the sum
$\sum_{\bfj\in\Z^{*,3}}\sum_{\bfk\in\calK_\bfj} \iu (\bfk\cdot\bfomega) y_{-\bfj}^{-\bfk}  \Omega_\bfj^2 c^2 y_\bfj^\bfk $ vanishes, because the term for $(\bfj,\bfk)$
cancels with that for $(-\bfj,-\bfk)$.
We have the following bounds for the terms on the right-hand sides of \eqref{odeyjk-jk} and \eqref{odeyjk-jj}.

\begin{lemma}\label{lem:d-bound}
We have, for  $0\le \tau\le 1$,
$$
\Bigl|\sum_{\bfj\in\Z^{*,3}} \sum_{\bfk\in\calK_\bfj} (\bfk\cdot\bfomega)\, y_{-\bfj}^{-\bfk} (\tau)\, d_\bfj^\bfk(\tau) \Bigr| \le
C\eps^{5-1/N},
$$
where $C$ is independent of $\eps$ and $0\le\tau\le 1$, but depends on $N$.
\end{lemma}

\begin{proof} This bound follows immediately  with the Cauchy--Schwarz inequality and using the bounds of Lemmas~\ref{lem:boundsz-3} and~\ref{lem:defect-3}. \qed
\end{proof}

\begin{lemma}\label{lem:e-bound} We have, for $0\le \tau\le 1$,
$$
\Bigl|\sum_{s\in\{-1,1\}}\sum_{\bfj\in\Z^{*,3}} s\Omega_\bfj\, y_{-\bfj}^{-s\jvec}(\tau)\,e_\bfj^{s\jvec}(\tau) \Bigr| \le C\eps^{5-1/N},
$$
where $C$ is independent of $\eps$ and $0\le\tau\le 1$, but depends on $N$.
\end{lemma}

\begin{proof} Consider first those terms in the sum defining $e_\bfj^{\pm\jvec}$ where one of the $\bfk_i$ ($i=1,2,3$) is different from $\pm\langle\bfj_i\rangle$. These terms yield a contribution of magnitude $\bigo(\eps^{5+\alpha})$ (with $\alpha=1/N$)
by the bounds of Lemma~\ref{lem:boundsz-3}. Hence it remains to bound
 $\sum_{s\in\{-1,1\}}\sum_{\bfj\in\Z^{*,3}} s\Omega_\bfj\,y_{-\bfj}^{-s\jvec}(\tau)\,\widehat e_\bfj^{s\jvec}(\tau)$, where
\begin{align*}
\widehat e_\bfj^{\;s\jvec} &=  -a  \sum_{\bfj_1+\bfj_2+\bfj_3=\bfj} 
 \sum_{ (s_1,s_2,s_3) \atop | s_1\Omega_{\bfj_1}+s_2\Omega_{\bfj_2}+s_3\Omega_{\bfj_3}-s\Omega_{\bfj}| < \eps^{1-\alpha} } 
y_{\bfj_1}^{s_1\langle\bfj_1\rangle}y_{\bfj_2}^{s_2\langle\bfj_2\rangle}y_{\bfj_3}^{s_3\langle\bfj_3\rangle},
\end{align*}
where the sum is over $s_i\in\{-1,1\}$ with the stated property.
We then have, on formally setting $\bfj_4=-\bfj$ and $s_4=-s$, and on using the symmetry of the expression in the second line,
\begin{align*}
&\sum_{s\in\{-1,1\}}\sum_{\bfj\in\Z^{*,3}} s\Omega_\bfj \, y_{-\bfj}^{-s\jvec}(\tau)\,\widehat e_\bfj^{s\jvec}(\tau) 
\\
&=
 a \sum_{\bfj_1+\bfj_2+\bfj_3+\bfj_4=\bfzero} 
\sum_{(s_1,s_2,s_3,s_4) \atop | s_1\Omega_{\bfj_1}+s_2\Omega_{\bfj_2}+s_3\Omega_{\bfj_3}+s_4\Omega_{\bfj_4}| < \eps^{1-\alpha} } \!\!\!\!\!\!\!
s_4\Omega_{\bfj_4}\,y_{\bfj_1}^{s_1\langle\bfj_1\rangle}y_{\bfj_2}^{s_2\langle\bfj_2\rangle}y_{\bfj_3}^{s_3\langle\bfj_3\rangle}y_{\bfj_4}^{s_4\langle\bfj_4\rangle}
\\
&=\frac{a}4\sum_{\bfj_1+\bfj_2+\bfj_3+\bfj_4=\bfzero} 
\sum_{(s_1,s_2,s_3,s_4) \atop | s_1\Omega_{\bfj_1}+s_2\Omega_{\bfj_2}+s_3\Omega_{\bfj_3}+s_4\Omega_{\bfj_4}| < \eps^{1-\alpha} } \!\!
\Bigl(\sum_{i=1}^4 s_i\Omega_{\bfj_i}\Bigr)\,\prod_{i=1}^4 y_{\bfj_i}^{s_i\langle\bfj_i\rangle}.
\end{align*}
Since $\bigl|\sum_{i=1}^4 s_i\Omega_{\bfj_i}\bigr|<\eps^{1-\alpha}$  and $|\bfy|_{\text{diag}}=\bigo(\eps)$, it follows with the Cauchy-Schwarz inequality and Lemma~\ref{lem:boundg-3} that this expression is $\bigo(\eps^{1-\alpha}\eps^4)$, which yields the result.
\qed
\end{proof}

\begin{theorem}\label{thm:calI-3}
Consider the expression
\[
\calI (\bfy  , \dot \bfy)= \eps \sum_{\bfj\in\Z^{*,3}} \sum_{\bfk\in\calK_\bfj}   \iu (\bfk\cdot\bfomega) y_{-\bfj}^{-\bfk} \dot y_\bfj^\bfk .
\]
In the situation of Theorem~\ref{thm:mfe-3},
the functions $y_\bfj^\bfk (\tau ) = z_\bfj^\bfk (\tau ) \, \e^{\iu (\bfk\cdot\bfomega) \phi (\tau )/\eps} $ 
then satisfy, for $0\le \eps t \le 1$,
\begin{equation}\label{thmest1-3}
\frac {\d}{\d t} \calI \bigl( \bfy (\eps t) , \dot \bfy (\eps t) \bigr) = \bigo (\eps^{5-1/N})
\end{equation}
and
\begin{equation}\label{thmest2-3}
\calI \bigl( \bfy (\eps t) , \dot \bfy (\eps t) \bigr)
= 2c(\eps t) \sum_{\bfj\in\Z^{*,3}} \Omega_\bfj^2 \big|z_\bfj^\jvec(\eps t) \big|^2 + \bigo (\eps^3 ) . 
\end{equation}
The constant symbolised by $\bigo (\cdot )$ depends on the truncation index $N$,
but it is independent of $0<\eps \le \eps^*$ (with $\eps^*$ sufficiently
small) and of $t$ as long as $0\le \eps t \le 1$.
\end{theorem}

\begin{proof} The bound \eqref{thmest1-3} is obtained by
differentiation of $\calI \bigl( \bfy (\eps t) , \dot \bfy (\eps t) \bigr)$
with respect to $t$ and using \eqref{U-term-vanish} and Lemmas~\ref{lem:d-bound} and~\ref{lem:e-bound}.
The relation~\eqref{thmest2-3} is proved in the same way as the analogous relation in Theorem~\ref{thm:calI}, using the bounds of Lemma~\ref{lem:boundsz-3}.
\qed
\end{proof}

\subsection{Connection with the action of the wave equation}

We consider the harmonic energy divided by the wave speed along the MFE $\widetilde u(\bfx,t)$ of Theorem~\ref{thm:mfe-3},
$$ 
\begin{array}{rcl}
\widetilde I(t) &=& \displaystyle \frac 1{2c(\eps t)} \Bigl( \| \partial_t \widetilde u(\cdot,t) \|^2 + 
c(\eps t)^2 \|\nabla_\bfx \widetilde u(\cdot,t) \|^2 \Bigr) \\[4mm]
&=& \displaystyle  \frac 1{2c(\eps t)} \Bigl( \sum_{\bfj\in\Z^{*,3}} \bigl|\frac\d{\d t} \widetilde u_\bfj(t)\bigr|^2 +c(\eps t)^2 
\sum_{\bfj\in\Z^{*,3}} \Omega_\bfj^2 | \widetilde u_\bfj(t)|^2  \Bigr).
\end{array}
$$ 

The following result is proved in the same way as Lemma~\ref{lem:action}.

\begin{lemma}\label{lem:action-3}
Let $\widetilde u(\bfx,t)$ be the MFE of Theorem~\ref{thm:mfe-3}. Then,
\[
\widetilde I(t)
= 2c(\eps t) \sum_{\bfj\in\Z^{*,3}} \Omega_\bfj^2 \big|z_\bfj^\jvec(\eps t) \big|^2 + \bigo (\eps^3 ) . 
\]
\end{lemma}

\subsection{Transitions in the almost-invariant}
The following result is obtained by the same arguments as in the proof of Lemma~\ref{lem:trans}, using the construction of the MFE in Section~\ref{sec:mfe-t-3}.

\begin{lemma}\label{lem:trans-3}
 Under the conditions of Theorem~\ref{thm:mfe-3}, let $z_\bfj^\bfk(\tau)$ for $0\le\tau=\eps t\le 1$ be the coefficient functions of the MFE as in Theorem~\ref{thm:mfe-3} for initial data $(u(\cdot,0),\partial_t u(\cdot,0))$, and let $y_\bfj^\bfk(\tau)=z_\bfj^\bfk(\tau)\e^{\iu (\bfk\cdot\bfomega)\phi(\tau)/\eps} $ and $\bfy(\tau)=\bigl( y_\bfj^\bfk(\tau)\bigr)$. Let further
 $\widetilde\bfy(\tau)=\bigl( \widetilde y_\bfj^\bfk(\tau)\bigr)$ be the corresponding functions of the MFE for $1\le \tau\le 2$ to the initial data $(u(\cdot,\eps^{-1}),\partial_t u(\cdot,\eps^{-1}))$, constructed as in Theorem~\ref{thm:mfe-3}. Then,
 $$
 \big|\calI\bigl(\bfy(1),\dot\bfy(1)\bigr) - \calI\bigl(\widetilde\bfy(1),\dot{\widetilde\bfy}(1)\bigr)\big|\le C\eps^{4-1/N},
 $$
 where $C$ is independent of $\eps$.
\end{lemma}

\subsection{Long-time conservation of the adiabatic invariant}
\label{subsec:patch-3}
In the same way as in Section~\ref{subsec:patch} we obtain from Theorem~\ref{thm:calI-3} and Lemmas~\ref{lem:action-3} and~\ref{lem:trans-3} that for $t\le \kappa_N\eps^{-3+1/N}$ with a sufficiently small $\kappa_N$,
$$
| I(t)-I(0)|\le C_1\eps^3+C_2\, t\, \eps^{5-1/N}   \le \eps^2.
$$
This yields the bound of Theorem~\ref{thm:main-t-3}.

\section*{Acknowledgement} We thank the referee for indicating a possible alternative proof of the result for the one-dimensional case (Theorem~\ref{thm:main-t}) by using techniques of Hamiltonian perturbation theory as developed by Neishtadt and Bambusi \& Giorgilli.

\providecommand{\bysame}{\leavevmode\hbox to3em{\hrulefill}\thinspace}
\providecommand{\MR}{\relax\ifhmode\unskip\space\fi MR }
\providecommand{\MRhref}[2]{%
  \href{http://www.ams.org/mathscinet-getitem?mr=#1}{#2}
}
\providecommand{\href}[2]{#2}

\end{document}